\documentclass[12pt]{amsart}
\usepackage{amsrefs}
\usepackage{amssymb}
\usepackage{amsmath}
\usepackage{graphicx}
\usepackage{mathrsfs}
\usepackage{tikz-cd}
\usepackage{enumerate}

\usepackage{fullpage}

\usepackage{cleveref}

\title{V*-algebras}
\author{Andre Kornell}
\thanks{The research reported here was supported by National Science Foundation grant DMS-1066368.}

\newtheorem{theorem}{Theorem}[section]
\newtheorem{lemma}[theorem]{Lemma}
\newtheorem{proposition}[theorem]{Proposition}
\newtheorem{corollary}[theorem]{Corollary}

\newtheorem{principle}[theorem]{Principle}

\theoremstyle{definition}
\newtheorem{definition}[theorem]{Definition}

\theoremstyle{remark}
\newtheorem{remark}[theorem]{}

\newtheorem{example}[theorem]{Example}

\theoremstyle{plain}
\newtheorem*{theorem*}{Theorem}
\newtheorem*{lemma*}{Lemma}
\newtheorem*{proposition*}{Proposition}

\theoremstyle{definition}
\newtheorem*{definition*}{Definition}

\theoremstyle{remark}

\newcommand{\<}{\langle}
\renewcommand{\>}{\rangle}
\renewcommand{\:}{\colon}
\newcommand{\subsetof}{\subseteq}
\newcommand{\To}{\rightarrow}
\newcommand{\tensor}{\otimes}
\newcommand{\iso}{\cong}
\newcommand{\inv}{^{-1}}

\newcommand{\supp}{\mathop{\mathrm{supp}}}
\newcommand{\sa}{\vphantom{)}_{sa}}

\newcommand{\B}{\mathcal B}

\newcommand{\F}{\mathcal F}

\renewcommand{\H}{\mathcal H}

\newcommand{\K}{\mathcal K}
\renewcommand{\L}{\mathcal L}

\newcommand{\N}{\mathcal N}

\renewcommand{\S}{\mathcal S}

\newcommand{\X}{\mathcal X}
\newcommand{\Y}{\mathcal Y}

\newcommand{\CC}{\mathbb C}

\newcommand{\II}{\mathbb I}

\newcommand{\NN}{\mathbb N}

\newcommand{\QQ}{\mathbb Q}
\newcommand{\RR}{\mathbb R} 

\newcommand{\TT}{\mathbb T}

\newcommand{\ZZ}{\mathbb Z}

\newcommand{\CCC}{\mathbf C}
\newcommand{\DDD}{\mathbf D}

\newcommand{\LLL}{\mathbf L}
\newcommand{\MMM}{\mathbf M}

\newcommand{\VVV}{\mathbf V}

\newcommand{\Bb}{\mathscr B}

\newcommand{\mmm}{\mathbf m}
\newcommand{\nnn}{\mathbf n}

\newcommand{\ppp}{\mathbf p}
\newcommand{\qqq}{\mathbf q}
 
\newcommand{\sss}{\mathbf s}
\newcommand{\ttt}{\mathbf t}

\begin{document}

\maketitle

\begin{abstract}
What is the correct noncommutative generalization of the functor $C_0(X) \mapsto \ell^\infty(X)$ for locally compact Hausdorff $X$ having a countable basis? Making the ansatz $\mathcal K(\ell^2) \mapsto \mathcal B(\ell^2)$, we expect that every unital $*$-homomorphism $C(\mathbb T) \rightarrow \mathcal B(\ell^2)$ extend canonically to a unital $*$-homomorphism $\ell^\infty(\mathbb T) \rightarrow \mathcal B(\ell^2)$. Thus, we expect to extend the continuous functional calculus for a unitary operator on $\ell^2$ to all bounded complex-valued functions.

Therefore, we work in a model of set theory where every set of real numbers is Lebesgue measurable; we must assume the consistency of an inaccessible cardinal in order to do so. The axiom of choice necessarily fails in such a model, but our model is carefully chosen to enable the verification of many familiar theorems via a scrutinization of their statements rather than their proofs. This technique significantly lowers the cost of doing interesting mathematics in this unfamiliar setting, and it is explained in detail.
  
By analogy with the ultraweak topology, we define the continuum-weak topology on bounded operators to be the topology given by functionals of the form $x \mapsto \int_0^1 \langle \eta_t| x \xi_t\rangle \, dt$. We then define a V*-algebra to be a C*-algebra of bounded operators that is closed in the continuum-weak topology. Every C*-algebra $A$ has an enveloping V*-algebra $V^*(A)$, and if $X$ is a locally compact Hausdorff space with a countable basis, then $V^*(C_0(X)) \cong \ell^\infty(X)$. More generally, if $A$ is any separable C*-algebra of type I, then $V^*(A)$ is canonically isomorphic to an $\ell^\infty$-direct sum of type I factors, with one summand for each irreducible representation of $A$. The self-adjoint part of any unital separable C*-algebra is isomorphic to the Banach space of strongly affine real-valued functions on its state space.
\end{abstract}

\section{Introduction}

\subsection*{Motivation}Every commutative unital C*-algebra is canonically isomorphic to the C*-algebra of continuous complex-valued functions on some compact Hausdorff space $X$. Noncommutative mathematics treats noncommutative unital C*-algebras as though they are also of this form, albeit no noncommutative unital C*-algebra is of this form for any $X$ that is a compact Hausdorff space in the literal sense. Hence, one imagines ``quantum'' compact Hausdorff spaces. Likewise, one imagines ``quantum'' continuous functions corresponding to unital $*$-homomorphisms between unital C*-algebras.

There are many categories of ``quantum'' objects that generalize familiar categories in this way, and it is striking that there is no widely accepted such generalization of the category of sets and functions. The principal circumstance is that an obvious such generalization exists, but it is disappointing in many ways. This obvious generalization is the opposite of the category whose objects are von Neumann algebras generated by their minimal projections, and whose morphisms are unital ultraweakly continuous $*$-homomorphisms, or equivalently, the opposite of the category whose objects are C*-algebras of compact operators, and whose morphisms are C*-algebra morphisms in the sense of Woronowicz \cite{Woronowicz79}. Unfortunately, this line of reasoning eventually leads to a conflict of notions between the resulting ``quantum'' functions, and the ``quantum'' continuous functions of the preceding paragraph.

This paper is motivated by the problem of finding a quantum generalization of the construction $C(X) \mapsto \ell^\infty(X)$. Here, $C(X)$ denotes the C*-algebra of continuous complex-valued functions on a compact Hausdorff space $X$, and $\ell^\infty(X)$ denotes the von Neumann algebra of all bounded complex-valued functions on $X$. The von Neumann algebra $\ell^\infty(X)$ and the inclusion $*$-homomorphism $C(X) \hookrightarrow \ell^\infty(X)$ are defined up to isomorphism by the following universal property:
$$
\begin{tikzcd}
C(X) \arrow[hook]{r} \arrow{rd} & \ell^\infty(X)
\arrow[dotted]{d}{!}\\
& \ell^\infty(Y)
\end{tikzcd}
$$
every unital $*$-homomorphism $C(X) \To \ell^\infty(Y)$, for some set $Y$, factors uniquely through the inclusion via a unital ultraweakly continuous $*$-homomorphism. The problem is to find a subcategory $\DDD$ of C*-algebras such that every unital C*-algebra $A$ has an enveloping algebra in $\DDD$ in the sense of the universal property just described, and such that the enveloping algebra of $C(X)$ is still $\ell^\infty(X)$. The object in the opposite of $\DDD$ corresponding to the enveloping algebra of a unital C*-algebra $A$ could then be viewed as its quantum spectrum, and the opposite of $\DDD$ would behave like a category of quantum sets and functions in at least this respect.

The present approach is motivated by two research topics. The first is the classical work of Davies \cite{Davies68} and Pedersen \cite{Pedersen69} on the quantum generalization of the construction $C(X) \mapsto \Bb(X)$, where $\Bb(X)$ denotes the C*-algebra of bounded complex-valued Baire functions on the compact Hausdorff space $X$. Davies defines the $\sigma$-envelope $A^{\sim}$, which is the enveloping $\mathrm{\Sigma^*}$-algebra of a separable C*-algebra $A$. Similarly, Pedersen defines the enveloping Borel $*$-algebra $\Bb(A)$:
$$
\begin{tikzcd}
A \arrow[hook]{r} \arrow {rd} &\Bb(A)
\arrow[dotted]{d}{!} \\
& B
\end{tikzcd}
$$
Notably, $\Bb(A) = A^{\sim}$ when $A$ is type I \cite{Pedersen69}*{theorem 8}, but it is not known if this equality holds generally.

Our second motivating research topic is the recent approach of generalizing sets directly to von Neumann algebras. This idea appears implicitly in Kuperberg and Weaver's research on quantum metrics \cite{KuperbergWeaver11}, and in Weaver's consequent distillation of quantum relations \cite{Weaver11}. The author continued this research by relating Weaver's quantum relations to the category of von Neumann algebras and unital ultraweakly continuous $*$-homomorphisms \cite{Kornell11}, and by showing that the opposite of this category shares several key properties with that of sets and functions \cite{Kornell12}. Overall, the approach is remarkably successful at interpreting familiar operator theoretic results as statements about ``quantum collections'', i. e., the objects of this opposite category. Every unital C*-algebra $A$ has an enveloping von Neumann algebra $W^*(A) = A^{**}$:
$$
\begin{tikzcd}
A \arrow[hook]{r} \arrow{rd} & W^*(A) 
\arrow[dotted]{d}{!}\\
& M
\end{tikzcd}
$$
The enveloping von Neumann algebra generalizes the construction $C(X) \mapsto L^\infty(X)$, where $L^\infty(X)$ denotes the category-theoretic limit of measure algebras $L^\infty(X, m)$ over probability measures $m$ on the compact Hausdorff space $X$. While the enveloping Borel $*$-algebra $\Bb(C(X))$ consists of only the Baire bounded complex-valued functions on $X$, the enveloping von Neumann algebra $W^*(C(X))$ consists of \emph{more than} all the bounded complex-valued functions; $\ell^\infty(X)$ is typically a proper subalgebra of $W^*(C(X))$. This difficulty with the category of von Neumann algebras and unital ultraweakly continuous $*$-homomorphisms is quite general; recently, Heunen and Reyes have shown \cite{HeunenReyes14B} that any functor $F$, from the category of unital C*-algebras and unital $*$-homomorphisms, to the category of von Neumann algebras and unital ultraweakly continuous $*$-homomorphisms, equipped with a natural transformation from the identity functor that is isomorphic to $C(X) \hookrightarrow \ell^\infty(X)$ at each commutative unital C*-algebra $C(X)$, must take $\B(\ell^2(\NN))$ to $0$. 

The study of various notions of spectrum for noncommutative C*-algebras is of course almost as old as the study of the C*-algebras themselves \cite{Kaplansky51}. The study of \emph{noncommutative} spectra began somewhat later with noncommutative general topology, starting with Akemann's definition of open projections \cite{Akemann69}. Giles and Kummer's introduction of q-spaces \cite{GilesKummer71} is particularly notable for its emphasis on the noncommutative generalization of \emph{sets}. I include references to several other approaches that have the noncommutative viewpoint in mind \cite{Mulvey86} \cite{BorceuxRosickyVandenBossche89} \cite{IshamButterfield98} \cite{HeunenLandsmanSpitters09} \cite{HeunenReyes14A}.

\subsection*{Results}

We work in a model of set theory with the property that every subset of $\RR$ is Lebesgue measurable. By Vitali's theorem, the axiom of choice necessarily fails in such a model, but we do retain two limited choice principles, which are sufficient for a large fragment of functional analysis. We avoid the enormous task of establishing this claim by requiring that our model be transitive and closed under countable unions; this enables the verification of a familiar mathematical result via the scrutiny of its statement, rather than its proof. We expound this approach.

For every Hilbert space $\H$, we define the continuum-weak topology on the bounded operator algebra $\B(\H)$ to be the initial topology for functionals of the form $x \mapsto \int_0^1 \< \eta_t | x| \xi_t\> \,dt$, a variation on the usual definition of the ultraweak topology. We then define a V*-algebra to be a concrete C*-algebra $E \subsetof \B(\H)$ that is closed in the continuum-weak topology. The representation theory of V*-algebras is another variant of the usual theories for C*-algebras and von Neumann algebras.

Every $*$-homomorphism from a C*-algebra $A$ to a V*-algebra $E$ factors uniquely through the continuum-weak closure of $A$ in its universal representation, via a continuum-weakly continuous $*$-homomorphism:
$$
\begin{tikzcd}
A \arrow{r}{\gamma_\S} \arrow{rd}[swap]{\rho} & V^*(A) \arrow[dotted]{d}{!}[swap]{\pi} \\
& E 
\end{tikzcd}
$$
If $A$ is unital, then so is $V^*(A)$; if $\rho$ is unital, then so is $\pi$; thus, $V^*(A)$ is the enveloping V*-algebra of $A$ in the sense above. The reader is cautioned that the universal representation $\gamma_\S$ need not be faithful in general. However, if $A$ is separable, then the universal representation is always faithful, and if $A$ is furthermore unital, then the self-adjoint elements of $V^*(A)$ correspond bijectively with the strongly affine real-valued functions on the state space $\S(A)$ \cite{LudvikSpurny}*{section 2}.

In our model of set theory, a separable C*-algebra is type I iff $\hat A \preccurlyeq \RR$, in the sense of cardinality. This is essentially a reformulation of Mackey's conjecture, proved by Glimm \cite{Glimm61}, that $A$ is type I iff $\hat A$ is a standard Borel space. The GNS construction yields a surjection from the pure state space $\partial \S(A) \preccurlyeq \RR$, onto the spectrum $\hat A$, but without the axiom of choice, there may not be an injection in the other direction. This initially alarming circumstance allows the reformulation of classifiability results as statements about cardinality, arguably their most elegant and natural form.

If $A$ is a commutative separable C*-algebra, then it is isomorphic to $C_0(X)$ for some second countable locally compact Hausdorff space $X$, by Gelfand duality. Its enveloping V*-algebra is then isomorphic to $\ell^\infty(X)$. This result generalizes naturally to separable C*-algebras of type I: if $A$ is such a C*-algebra, then the inclusion $A \hookrightarrow V^*(A)$ is isomorphic to the direct sum of all type I factor representations $\gamma\: A \mapsto \gamma(A)''$, up to the natural notion of equivalence. This result is akin to the theorem of Davies \cite{Davies68}*{theorem 4.5} regarding the structure of the $\sigma$-envelope of a type I separable C*-algebra.

Thus, the category of unital V*-algebras and unital continuum-weakly continuous $*$-homomorphisms is a subcategory of the category of unital C*-algebras and unital $*$-homomorphisms, every unital C*-algebra has an enveloping unital V*-algebra, and $V^*(C(X)) = \ell^\infty(X)$ for every \emph{second countable} compact Hausdorff space $X$. The statement of the problem asks for such a subcategory, but it asks that the enveloping algebra of $C(X)$ be $\ell^\infty(X)$ for \emph{every} compact Hausdorff $X$. This is not the case for the subcategory of unital V*-algebras. If $X \succcurlyeq \RR$ is a discrete space, then $V^*(c_0(X)+ \CC)$, the enveloping V*-algebra of the C*-algebra of complex-valued functions on $X$ converging at infinity, is not isomorphic to $\ell^\infty(X)$.

\subsection*{Nonunital algebras}

For simplicity, our exposition has been framed in terms of unital C*-algebras, but nonunital C*-algebras play an important role in noncommutative mathematics. Assuming the axiom of choice, every commutative C*-algebra is isomorphic to $C_0(X)$, the C*-algebra of continuous complex-valued functions on some locally compact Hausdorff space $X$ that vanish at infinity, so noncommutative C*-algebras can be viewed as quantum locally compact Hausdorff spaces. The correct notion of morphism for nonunital C*-algebras, from the point of view of noncommutative geometry, is that of Woronowicz \cite{Woronowicz79}; we will refer to these morphisms as C*-morphisms. A C*-morphism from a C*-algebra $A$ to a C*-algebra $B$ is a $*$-homomorphism $\pi$ from $A$ to the centralizer C*-algebra $M(B)$ such that $\pi(A)B$ is norm dense in $B$. The opposite of the category of commutative C*-algebras and C*-morphisms is equivalent to the category of locally compact Hausdorff spaces and continuous functions.

The enveloping V*-algebra of a nonunital C*-algebra may or may not be unital. For example, $V^*(C_0(\RR)) = \ell^\infty(\RR)$ is unital, but if $X$ is $\RR$ in its discrete topology, then $V^*(C_0(X))$ consists of the countably supported bounded complex-valued functions on $X$. Thus, it is natural to define a morphism of $V^*$-algebras to be a continuum-weakly continuous $*$-homomorphism $\pi: E \To M(F)$ such that $\pi(E)F$ is continuum-weakly dense in $F$. Fortunately, the centralizer algebra $M(F)$ is always itself a V*-algebra. If the locally compact Hausdorff space $X$ is locally Polish, in the sense that every point has a neighborhood which is a Polish space in the subspace topology, then $M(C_0(X)) \iso \ell^\infty(X)$. In particular, $M(V^*(C_0(X))) \iso \ell^\infty(X)$ for every discrete space $X$.

We can reformulate the universal property of enveloping V*-algebras in terms of morphisms: Let $A$ be a C*-algebra and let $E$ be a V*-algebra. If $\rho: A \To M(E)$ is a $*$-homomorphism such that $\rho(A)E$ is continuum-weakly dense in $E$, then there is a unique continuum-weakly continuous $*$-homomorphism $\pi: V^*(A) \To M(E)$ such that $\pi(V^*(A))E$ is continuum-weakly dense in $E$, and $\pi \circ \gamma_\S = \rho$:
$$
\begin{tikzcd}
A \arrow{r}{\gamma_\S} 
\arrow{rd}[swap]{\rho}
&
V^*(A) \arrow[dotted]{d}{!}[swap]{\pi}
\\
&
M(E)
\end{tikzcd}
$$
Defining a V*-morphism from a V*-algebra $E$ to a V*-algebra $F$ to be a continuum-weakly continuous $*$-homomorphism $\pi$ from $E$ to $M(F)$ such that $\pi(E)F$ is continuum-weakly dense in $F$, we find that every C*-morphism from a C*-algebra $A$ to a C*-algebra $B$ extends to a V*-morphism from the enveloping V*-algebra $V^*(A)$ to the enveloping V*-algebra $V^*(B)$:
$$
\begin{tikzcd}
A \arrow{r} \arrow{d}{\gamma_{\S(A)}} & M(B) \arrow{d}{\tilde \gamma_{\S(B)}} \\
V^*(A) \arrow[dotted]{r}{!} & M(V^*(B))
\end{tikzcd}
$$
Here, $\gamma_{\S(A)}$ denotes the universal representation of $A$, and $\tilde \gamma_{\S(A)}$ is the unique extension of the universal representation of $B$ to the centralizer C*-algebra $M(B)$. Thus, we obtain a functor from the opposite of the category of C*-algebras and C*-morphisms, to the opposite of the category of V*-algebras and V*-morphisms, which might be viewed as a notion of quantum spectrum, but although $V^*(C_0(\RR)) \iso \ell^\infty(\RR)$, the V*-algebra $V^*(C_0(\RR))$ cannot be identified with the set $\RR$, because it is not isomorphic to $V^*(C_0(X))$ for $X$ a discrete space equinumerous to $\RR$. 

\subsection*{Quantum spectrum}

Another disadvantage of the enveloping V*-algebra as a notion of quantum spectrum is that there are unital C*-algebras $A$ such that $M(V^*(A)) = V^*(A)$ is not a von Neumann algebra. Our first example: the enveloping V*-algebra $V^*(C^*_r(F_2))$ of the reduced group C*-algebra of the free group on two generators is not a von Neumann algebra. Our second example: the enveloping V*-algebra $V^*(c_0(\RR) + \CC)$ of the C*-algebra of all complex-valued functions on $\RR$ converging at infinity is not a von Neumann algebra. In each case, the conclusion follows from the existence of two disjoint families of irreducible representations of $A$ whose direct integrals are equivalent. For the C*-algebra $C^*_r(F_2)$, this is a classical result of Yoshizawa \cite{Yoshizawa51}. For the C*-algebra $c_0(\RR) + \CC$, the irreducible representation whose kernel is $c_0(\RR)$ is equivalent to the direct integral of point evaluations for any atomless probability measure on $\RR$.

A variant of this second example illuminates the situation: The group of rational rotations acts ergodically on the circle $\TT = \RR/\ZZ$, which yields a unital continuum-weakly continuous $*$-homomorphism $\gamma\: \ell^\infty(\TT/\QQ) \To \CC$. If $X$ is $\TT/\QQ$ as a discrete space, and $\tilde X$ is its one-point compactification, then evaluation at infinity $C(\tilde X) \To \CC$ extends to a unital continuum-weakly continuous $*$-homomorphism $\ell^\infty(\tilde X) \To \CC$ in at least two distinct ways: evaluation at infinity, and via $\gamma$. Thus, we are reminded that only the unital ultraweakly continuous $*$-homomorphisms $\ell^\infty(X) \To \ell^\infty(Y)$ correspond to functions from $Y$ to $X$; the unital continuum-weakly continuous $*$-homomorphisms do not all correspond to functions.

A careful consideration of this circumstance suggests a somewhat coherent narrative for noncommutative general topology in its current form. Quantum sets correspond to von Neumann algebras whose center is of the form $\ell^\infty(X)$ for some set $X$, i. e., to direct sums of factors. Continuum-weakly continuous states correspond to quantum probability measures, and ultraweakly continuous states correspond to quantum atomic probability measures. Unital ultraweakly continuous $*$-homomorphisms correspond to quantum functions, and unital continuum-weakly continuous completely positive maps correspond to quantum probabilistic functions, which assign probability measures on the codomain to the imagined elements of the domain. Of these, the unital continuum-weakly continuous $*$-homomorphisms correspond to quantum probabilistic functions which are almost deterministic; a unital continuum-weakly continuous $*$-homomorphism $\ell^\infty(X) \To \CC$ corresponds to a probability measure that may not be evaluation at an element of $X$, but nevertheless, the variance of each observable on $X$ is intuitively infinitesimal, and literally zero.

If $A$ is a C*-algebra, let $\gamma_\F: A \To F(A)$ denote the direct sum of its factor representations, with $F(A) = \gamma_\F(A)''$. If $A$ admits a disintegration theory, in the sense that the universal representation $\gamma_\S: A \To W^*(A)$ factors through $\gamma_\F$ via some unital continuum-weakly continuous $*$-homomorphism, then every $*$-homomorphism $\pi$ from $A$ to some von Neumann algebra $M$, with $\pi(A)M$ ultraweakly dense in $M$, factors through $\gamma_\F$ via some continuum-weakly continuous $*$-homomorphism:
$$
\begin{tikzcd}
A \arrow{r}{\gamma_\F} \arrow{rd} &
F(A) \arrow[dotted]{d} \\
&
M
\end{tikzcd}
$$
In this case, $F(A)$ corresponds to the quantum spectrum of $A$. It follows that if $A$ and $B$ are C*-algebras that admit a disintegration theory, then each $*$-homomorphism $\pi\: A \To M(B)$, with $\pi(A)B$ norm dense in $B$, extends to unital continuum-weakly continuous $*$-homomorphism $F(A) \To F(B)$; thus, each such $*$-homomorphism is precomposition by a quantum probabilistic function from the quantum spectrum of $B$ to the quantum spectrum of $A$ that is almost deterministic.

Practical considerations compel me to make this draft available before I've had time to carefully consider this narrative, and before I've had time to add the proofs of a few relevant facts that seem more or less straightforward. I believe that every separable C*-algebra has a disintegration theory in the sense above. However, not every C*-algebra admits a disintegration theory, e. g., $L^\infty(\RR, \lambda)$. I don't know whether every C*-algebra of the form $C_0(X)$, for some locally compact Hausdorff space $X$, admits a disintegration theory. However, I believe that if $X$ is also locally Polish, in the sense that every point of $X$ has a neighborhood that is a Polish space in its subspace topology, then $C_0(X)$ does admit a disintegration theory. I believe that a locally compact Hausdorff space $X$ is locally Polish iff for every GNS representation $\gamma_\mu: C(X) \To \B(\H_\mu)$ there exists a separable ideal $I$ such that $\gamma_\mu(C(X)) = \gamma_\mu(I)$, and that any noncommutative C*-algebra with this property necessarily admits a disintegration theory. Thus, we might speak of quantum locally compact locally Polish spaces.

In this picture of noncommutative topology, only the liminary C*-algebras with a disintegration theory consist of functions. If $A$ admits a disintegration theory, but it is not liminary, then there is a self-adjoint element $a \in A$ of norm $1$ such that the functional calculus $C([0,1]) \To F(A)$ has no ultraweakly continuous extension $\ell^\infty([0,1]) \To F(A)$; thus $a$ does not correspond to a real-valued quantum function, only to an almost deterministic real-valued quantum probabilistic function.

It is tempting to replace direct sums of factors with direct sums of type I factors in the above narrative, because the latter are precisely the von Neumann algebras that are generated by their minimal projections, which are the ``obvious'' candidates for the quantum generalization of sets. And indeed, every cyclic representation of a separable C*-algebra disintegrates into irreducible representations. However, the axiom of choice fails in our model of set theory, and it seems probable that the universal representation of a separable C*-algebra need not disintegrate into irreducible representations.

\subsection*{Acknowledgements}

I am very grateful to my advisor, Marc Rieffel, for his support and encouragement, and for numerous valuable comments on earlier drafts of this paper. I am also indebted to Andrew Marks, John Steel, and Hugh Woodin, who have patiently answered my many questions regarding set theory. The approach taken in the present paper relies crucially on their input; it also incorporates several of their suggestions. I thank Leo Harrington for catching my omission of the word ``separable'' in item 123 of section 7. I thank Vaughn Jones for giving me the opportunity to present some of this research to the vibrant operator theory group at Vanderbilt. I also thank Ilijas Farah for giving me the opportunity to contribute a talk on this research to the East Coast Operator Algebras Symposium. Finally, I thank Arnaud Brothier, Lawrence Brown, Alexandru Chirvasitu, Michael Hartglass, Benjamin Hayes, Sridhar Ramesh, Ji\v{r}\'{i} Spurn\'{y}, and Noah Schweber, as well as everyone I mention above, for their comments and answers.

\subsection*{Notation and terminology}

The symbol $\II$ denotes the unit interval. The symbol $\TT$ denotes either the unit circle in the complex plane, or the quotient $\RR/\ZZ$, as appropriate to the context. A topological space is Polish iff it is homeomorphic to a complete separable metric space.

Let $f\: T \To S$ be a function. The expression $(f(t) \in S \: t\in T)$ also denotes the function $f$; the expression $\{f(t)\in S\: t \in T\}$ denotes its range; and the expression $[f(t) \in S \: t \in T]$ denotes its equivalence class under the relevant equivalence relation. Where there is no danger of ambiguity, we will sometimes denote these objects simply by $(f(t))$, $\{f(t)\}$, and $[f(t)]$.

Let $X$ be a locally compact Hausdorff space. The support of a function $f\: X \To \CC$ is the set $\mathrm{supp}(f) =\{x \in X\: f(x) \neq 0\}$, and $f$ is compactly supported iff the closure $\overline{\mathrm{supp}(f)}$ is compact. A function $f \: X \To \CC$ is in $\ell^\infty(X)$ iff it is bounded; it is in $\ell^2(X)$ iff it is square-summable; it is in $\ell^1(X)$ iff it is absolutely summable; it is in $c_0(X)$ iff $\{|f(x)|\geq \epsilon\}$ is finite for all $\epsilon >0$; it is in $C_0(X)$ iff it is continuous and $\{|f(x)|\geq \epsilon\}$ is compact for all $\epsilon >0$; it is in $C_c(X)$ iff it is continuous and compactly supported; and it is in $\Bb(X)$ iff its real and imaginary parts can be obtained from the continuous compactly supported real-valued functions on $X$ by taking limits of ascending and descending sequences. A subset $S \subsetof X$ is Baire iff the function $f\: X \To \{0,1\}$ such that $f\inv(1)=S$ is in $\Bb(X)$ \cite{Pedersen88}*{remark 6.2.10}.

Let $T$ be a set. A complex-valued measure on $T$ is a function $m \: \{S \subsetof T\} \To \CC$ that is countably additive in the sense that the series $\sum_n m(S_n)$ converges absolutely to $m(\bigcup_n S_n)$ for all countable families $(S_n\subsetof T)$ of pairwise disjoint subsets. A finite measure is a complex-valued measure $m$ such that $m(S) \geq 0$ for all $S \subsetof T$; it is a probability measure if furthermore $m(T) = 1$. If $T$ is equipped with a finite measure $m$, and $\X$ is a Banach space, then $\L^1( T, \X) = \{f\: T \To \X \: \int_{t \in T} \|f(t)\| \, dm < \infty\}$ is the set of integrable functions, $\L^2(T, \X)= \{f \: T \To \X\: \int_{t \in T} \|f(t)\|^2 \, dm < \infty \}$ is the set of square-integrable functions, and $\L^\infty(T, \X) = \{f\: T \To \X\: \inf_{m(S) = m(T)} \sup_{t \in S} \|f(t)\| < \infty  \}$ is the set of essentially bounded functions; and $L^1(T, \X)$, $L^2(T, \X)$, and $L^\infty(T, \X)$ are the corresponding Lebesgue spaces. The unit interval $\II$ will often be implicitly equipped with Lebesgue measure.

Let $A$ be a C*-algebra. A functional on $A$ is a bounded linear function $A \To \CC$. If $\mu$ is a state on $A$, then $\gamma_\mu\: A \To \B(\H_\mu)$ denotes the GNS representation of $A$ for the state $\mu$. If $T$ is a set of states, then $\gamma_T$ denotes the direct sum representation $a \mapsto \bigoplus_{\mu \in T}\gamma_\mu(a)$. The space of all states of $A$ is denoted $\S(A)$, and $\gamma_{\S(A)}$ is often abbreviated $\gamma_\S$; this is the universal representation. The space of all \emph{pure} states is denoted $\partial \S(A)$, and $\gamma_{\partial \S(A)}$ is often abbreviate $\gamma_\partial$; this is the atomic representation. Also, $W^*(A)$ is the closure of $\gamma_\S(A)$ in the ultraweak topology, and $V^*(A)$ is the closure of $\gamma_\S(A)$ in the continuum-weak topology, which is to be defined.

Let $A$ be a concrete C*-algebra on some Hilbert space $\H$. Then, $A$ is nondegenerate iff $A\H$ is norm dense in $\H$. A vector functional on $A$ is a functional of the form $ a \mapsto \< \xi| a \eta\>$ for some $\xi, \eta \in \H$. A vector state on $A$ is a state that is a vector functional, or equivalently, a state of the form $a \mapsto \< \xi| a \xi \>$ for some $\xi \in \H$. The Banach space of all ultraweakly continuous functionals is denoted $A_*$; the subspace of all ultraweakly continuous states is denoted $\S_*(A)$, and $\gamma_{\S_*(A)}$ is often abbreviated $\gamma_*$. The Banach space of all continuum-weakly continuous functionals is denoted $A_\bullet$; the subspace of all continuum-weakly continuous states is denoted $\S_\bullet(A)$, and $\gamma_{\S_\bullet(A)}$ is often abbreviated $\gamma_\bullet$. The expression $\overline A$ always denotes the closure of $A$ in the continuum-weak topology.

A V*-algebra is a concrete C*-algebra that is closed in the continuum-weak topology, which is to be defined. The term is intended to suggest an analogy with W*-algebras, i. e., von Neumann algebras. The term ``V*-algebra'' was previously used for a different class of operator algebras \cite{Berkson66}; our appropriation of this term might be justified by the fact that this class of operator algebras coincides with that of C*-algebras \cite{Berkson66}*{theorem 4.3}. The term ``v*-algebra'' is also used in universal algebra \cite{Urbanik63}.

A diagram with a dotted arrow expresses that there is a morphism which makes that diagram commute. A diagram with an exclamation mark expresses that the morphism making that diagram commute is unique. A diagram with neither of these features simply expresses that that diagram commutes.

Pedersen's \emph{C*-algebras and their Automorphism Groups} is our basic reference \cite{Pedersen79} \cite{Pedersen88}. I also recommend Blackadar's \emph{Operator Algebras} \cite{Blackadar06}, Jech's \emph{Set Theory} \cite{Jech02}, and Schechter's \emph{Handbook of Analysis and its Foundations} \cite{Schechter97} as encyclopedic references for their respective subjects.

We implicitly assume the existence of a proper class of Woodin cardinals that are limits of Woodin cardinals. All assertions in \cref{part2} and subsequent sections are implicitly stated \emph{in the Chang model} $\CCC$, i. e., $\CCC^{\omega_1}$ \cite{Chang71}.

\section{The role of set theory}

Axiomatic set theory is the established foundational system for mathematics. There are foundational systems that are more category-theoretic in nature, but we are not interested in formal foundations in the present paper. We are interested in arguing that the assumption that every set of real numbers is Lebesgue measurable is consistent with a large fragment of modern mathematics. Thus, we discuss set theory because Zermelo-Fraenkel set theory, its extensions, and its fragments, are the objects of modern research into consistency.

The role of set theory as a foundational system for mathematics may be explained with a simile. Computers handle finite mathematical structures by storing their data as strings of bits. Conceptually, a finite graph is not literally a string of bits, but an ideal computer can search through finite graphs by searching through strings of the appropriate kind. Similarly, the set $\{\{\}, \{\{\}\}\}$ is not literally the number $2$, but it is a set that is commonly used to represent the number $2$. Thus, a set should be thought of as a block of information that can be used to represent a mathematical structure.

Set theory typically considers only hereditary sets, which are defined recursively as sets whose elements are all hereditary sets, starting with the empty set. The ``universe'' of all hereditary sets, denoted $\VVV$, is assumed to satisfy the Zermelo-Fraenkel axioms of set theory, denoted $\mathbf{ZF}$. These axioms formalize our intuition of sets as classes of limited size. The universe $\mathbf V$ is also assumed to satisfy the axiom of choice, denoted $\mathbf{AC}$. The theory $\mathbf{ZF +AC}$ is the usual foundation for mathematics, and mathematical objects are assumed to be represented by members of $\mathbf V$.

The \emph{large cardinal axioms} are additional set theoretic axioms which express the existence of large sets. They generalize the axiom of infinity, an axiom of $\mathbf{ZF}$, which expresses the existence of an infinite set. The large cardinal axioms expand the range of mathematical situations that can be provably modeled in $\VVV$. These axioms are not provably consistent with $\mathbf{ZF +AC}$, just as some of the axioms of $\mathbf{ZF +AC}$ are not provably consistent with weaker axiom systems that are nevertheless sufficient to formalize the vast majority of research mathematics.

\section{The interpretation of mathematics in set theory}

The details of the interpretation of mathematics in set theory are not important to us, but the general features of this interpretation are useful to have in mind. One should think of hereditary sets, the members of $\mathbf V$, as blocks of raw information that are being used to store the data of mathematical structures. We assume that every mathematical structure is encoded by some hereditary set.

The theory $\mathbf{ZF} + \mathbf{AC}$ is a foundation for mathematics in the sense that various properties of mathematical objects, and their tuples, can be expressed as properties of the hereditary sets that encode them, and mathematical reasoning can be expressed as deductions about these properties from the axioms of $\mathbf{ZF} + \mathbf{AC}$, or just from the axioms of $\mathbf{ZF}$ if no appeal to the axiom of choice is made. This is analogous to programming. The properties of hereditary sets are typically expressed in the \emph{language of set theory}. Formally, the language of set theory is the first-order language on a single binary predicate, written $\in$. The properties expressed by this language are those that can be written down using the symbols $\neg$, $\vee$, $\wedge$, $\forall$, $\exists$, $=$, $\in$, and a stock of variables. Sensible strings composed of these symbols are typically called formulas.

There is essentially one standard interpretation of mathematics in set theory: a standard way of associating hereditary sets to mathematical objects, a standard way of associating formulas of the language of set theory to mathematical properties, and a standard way of associating deductions of the theory $\mathbf{ZF+AC}$ to mathematical proofs. The interpretations that appear in various texts do differ, most prominently in the coding of the real numbers, but they are interchangeable for our purposes. For concreteness, we mention the interpretation provided  in the classic series of Bourbaki \cite{Bourbaki04}. We use the phrase ``the standard interpretation'' to describe a slightly modified version of this interpretation, which is conceptually cleaner. Specifically, in Bourbaki's interpretation, a hereditary set can code more than one object, e. g., the empty set codes both itself and the natural number zero, but it is simple to modify their interpretation to eliminate this phenomenon. 

The key point is that under the standard interpretation, the properties of the set theoretic universe $\VVV$ correspond to mathematical truths. A group is amenable iff the hereditary set that codes it has some property expressible in the language of set theory. This property is the existence of another hereditary set, a block of information that codes an invariant mean on the given group. When we restrict our attention exclusively to mathematical objects that are coded by hereditary sets that are constructible in a certain sense, the properties of such an object may be altered, as a result of excluding blocks of data describing possible ways to manipulate this object.

\section{The Chang model}

Intuitively, an inner model of set theory is a ``subuniverse'' of $\mathbf V$. Thus, an inner model $\mathbf M$ is a class of hereditary sets such that the Zermelo-Fraenkel axioms are true in $\MMM$, i. e., true if all variables are restricted to $\MMM$. Additionally, an inner model must be \emph{transitive}, and \emph{almost universal}. Transitivity is the condition that any element of a set in $\mathbf M$ must itself be in $\mathbf M$; thus, a member of $\mathbf M$ has exactly the same elements in $\mathbf M$ as it has in $\mathbf V$. Almost universality is the condition that any set of members of $\mathbf M$ is a subset of a member of $\mathbf M$.

The Chang model $\CCC$ is the smallest inner model of set theory that is closed under countable unions. This model first appeared in Chang's paper \cite{Chang71}; its existence is nontrivial. The Chang model $\CCC$ consists of all sets that are constructible in a certain sense, which is a modified version of the usual notion of set theoretic constructibility in the sense of G\"odel's constructible universe $\LLL$.

We assume that the set theoretic universe $\VVV$ satisfies the following large cardinal axiom: there exists a proper class of Woodin cardinals that are limits of Woodin cardinals. \textbf{Hugh Woodin} has shown that this large cardinal axiom is sufficient to establish that the Chang model $\CCC$ satisfies the axioms $\mathbf{DC}$, $\mathbf{AC_{ae}}$, and $\mathbf{AD^+}$ \cite{Woodin14} \cite{Larson}*{p. 44}. It follows that the Chang model does not satisfy $\mathbf{AC}$, so the Chang model $\CCC$ is strictly smaller than the set theoretic universe $\VVV$, and strictly larger than G\"odel's constructible universe $\LLL$.

The axiom of dependent choices, $\mathbf{DC}$, is a fragment of the axiom of choice. It says that it is possible to make a countable sequence of choices, where at each step one chooses from a set that depends on the choices made at previous steps. Formally, the axiom $\mathbf{DC}$ says that for any directed graph, if every vertex is the source of some arrow, then there exists a path through that directed graph from any given vertex that is either infinite or returns that vertex. It is equivalent to the Baire category theorem for complete metric spaces.

The axiom of choice almost everywhere, $\mathbf{AC_{ae}}$, is another fragment of the axiom of choice. It says that for any continuum of disjoint sets, it is possible to select an element from almost all of them. Formally, the axiom $\mathbf{AC_{ae}}$ says that for every family of sets $\{C_t\: t \in \II\}$, there exist a Lebesgue-measurable set $X \subsetof \II$ of full measure, and a function $f\: X \To \bigcup C_t$ such that $f(t) \in C_t$ for all $t \in X$.

The axiom $\mathbf{AD^+}$ is a technical strengthening of the axiom of determinacy, $\mathbf{AD}$. The axiom of determinacy says that for any game of perfect information, which consists of a countable sequence of moves made by two players from among countably many options, one of the two players has a winning strategy. It is sometimes rendered metaphorically as $\forall A \subsetof \NN^\NN\:$
\begin{align*}
\neg \exists n_0 \: \forall n_1\: & \exists n_2 \: \forall n_3\: \cdots (n_0, n_1, n_2, n_3, \ldots) \in A  \\ & \Longleftrightarrow 
 \forall n_0 \: \exists n_1\: \forall n_2 \: \exists n_3\: \cdots \neg (n_0, n_1, n_2, n_3, \ldots) \in A   
\end{align*}
We will not state the axiom of determinacy precisely; instead we list some of its known consequences under $\mathbf{ZF + DC}$:
\begin{enumerate}
\item $\mathbf{LM}$: Every subset of $\II$ is Lebesgue measurable.
\item $\mathbf{BP}$: Every subset of $\II$ has the Baire property.
\item $\mathbf{PSP}$: Every subset of $\II$ is countable or contains a perfect subset.
\item All ultrafilters are countably additive. All ultrafilters on $\NN$ are principal.
\item Every linear function between Fr\'echet spaces is continuous \cite{Garnir74}. Every linear function between Banach spaces is bounded.
\item For every Banach space, its algebraic dual is equal to its continuous dual.
For all $\sigma$-finite measure spaces $S$, $L^\infty(S)^* = L^1(S)$ \cite{Vath98}, and in particular $\ell^\infty(\NN)^* = \ell^1(\NN)$ \cite{Pincus74} .
\end{enumerate}

\subsection*{Cardinality}
Recall that for two sets $X$ and $Y$, we write $X \preccurlyeq Y$ in case there is an injection from $X$ into $Y$; this defines a preorder on the class of all sets. If $X \preccurlyeq Y$ and $Y \preccurlyeq X$, then $X \approx Y$, i. e., there exists a bijection between $X$ and $Y$; the proof of the Schr\"oder-Bernstein theorem does not use the axiom of choice. However, the axiom of choice is necessary to show that any two sets are comparable, and in the Chang model, neither $\omega_1 \preccurlyeq \RR$ nor $\RR \preccurlyeq \omega_1$.

Without the axiom of choice, the existence of a surjection of $X$ onto $Y$ does not necessarily imply the existence of an injection of $Y$ into $X$. In the Chang model, it is possible for a quotient of a set $X$ to be \emph{strictly larger} than $X$. For example, in the Chang model, there is no injection of $\TT/\QQ$ into $\TT$, because the latter has no nonprincipal ultrafilters, but the pushforward of Lebesgue measure along the quotient map $\TT \To \TT/\QQ$ yields such an ultrafilter on $\TT/\QQ$. On the other hand, for any enumeration $(q_n)$ of $\TT \cap \QQ$, the function $t \mapsto [{ \sum_{q_n \in [0,t]} \frac{1}{n!}}]$ is an injection of $\TT$ into $\TT/\QQ$.

\section{Consistency}

The arguments of the present paper can be carried out in any transitive model (a set or a proper class) of $\mathbf{ZF} + \mathbf{DC} + \mathbf{AC_{ae}} + \mathbf{LM} + \mathbf{BP} + \mathbf{PSP}$ that is closed under countable unions. The existence of such a transitive model is equiconsistent with the existence of an inaccessible cardinal, which is a very weak large cardinal axiom. \textbf{John Steel} pointed out to me that Solovay's model is such a transitive model in the relevant forcing extension, whenever the ground model satisfies the axiom of constructibilty, which may be trivially arranged \cite{Solovay70} \cite{KanoveiLambalgen08} \cite{Shelah84}. Thus, all of the results of the present paper that are stated in the Chang model also hold in this Solovay model.

An inaccessible cardinal is an infinite cardinal number $\kappa$ with the property that the set of smaller cardinals is closed under cardinal exponentiation, and under indexed sums, that is, under the operation $ \phi \mapsto \sum_{\alpha \in \lambda} \phi(\alpha)$, for $\phi = (\phi(\alpha) < \kappa \: \alpha \in \lambda)$ a family of cardinal numbers indexed by a cardinal number $\lambda < \kappa$. An inaccessible cardinal is also usually required to be uncountable to exclude the cardinal $\kappa = \aleph_0$, whose existence is one of the Zermelo-Fraenkel axioms $\mathbf{ZF}$; thus, one might justify the existence of inaccessible cardinals in roughly the same way that one justifies the existence of the set $\NN$ of natural numbers.

It is impossible to prove the existence of an inaccessible cardinal. If $\kappa$ is an inaccessible cardinal, then the set $V_\kappa$ of hereditary sets of rank  less than $\kappa$ is a model of $\mathbf{ZF}$, so the existence of an inaccessible cardinal implies the consistency of $\mathbf{ZF}$; thus, by G\"odel's incompleteness theorem we cannot establish the existence of such a cardinal from $\mathbf{ZF}$. A tiny variation on this argument shows that it is likewise impossible to establish the consistency of the existence of an inaccessible cardinal. However, the inconsistency of this large cardinal axiom would imply that a nonmeasurable subset of $\RR$ can be obtained without using the axiom of choice \cite{Hamkins11}. Furthermore, the existence of inaccessible cardinals is sometimes assumed in other branches of mathematics, for example, in the guise of Grothendieck universes.

The author's view is that we should accept any axiom that is convenient and that has been vetted by the set theoretic community, just as we do in the case of the theory $\mathbf{ZF +AC}$, i. e., $\mathbf{ZFC}$, the Zermelo-Fraenkel axioms with choice. Most of modern mathematics can be formalized in much weaker axiomatic systems such as Zermelo set theory with choice, or even third order arithmetic \cite{Weaver09}, so we do not accept $\mathbf{ZFC}$ out of necessity. Nevertheless, it is possible that these axioms are inconsistent. The very history of $\mathbf{ZFC}$ shows that its axioms are not self-evident. The notion that our century of working under $\mathbf{ZFC}$ is evidence for its consistency is particularly pernicious because most of that work relies on a tiny fraction of its power; this century of mathematical development is no more evidence for the consistency of $\mathbf{ZFC}$, than a century of elementary school mathematics is evidence for the consistency of Peano arithmetic. Our best evidence for the consistency of $\mathbf{ZFC}$ is set theoretical research, which consciously probes the boundaries of this axiomatic system, rather than consciously avoiding them. The overwhelming consensus among working set theorists is that the existence of an inaccessible cardinal is consistent with $\mathbf{ZFC}$.

Although the arguments of the present paper can be carried out in any transitive model of $\mathbf{ ZF + DC + AC_{ae} + LM + BP + PSP}$ that is closed under countable unions, the existence of which is equiconsistent with the existence of an inaccessible cardinal, we formally work specifically in the Chang model, and we assume the existence of a proper class of Woodin cardinals that are limits of Woodin cardinals, which is a much stronger large cardinal assumption. Before giving our reasons for doing so, we mention that it is possible to argue for the consistency of this assumption along the above lines. 

Unlike the model produced by Solovay's construction, the Chang model is canonical, and its definition is easy to intuit. The transitive model produced by Solovay's construction is an inner model of a forcing extension, whose definition is technical from the perspective of a functional analyst, and depends on the choice of a generic filter. On the other hand, the Chang model is the smallest inner model that is closed under countable unions, and it is the class of all constructible sets for language $L_{\omega_1 \omega_1}$, which is a variant of the language of set theory that allows countably infinite conjunctions and disjunctions such as $\phi_0 \vee \phi_1 \vee \ldots$, and countably infinite sequences of existential or universal quantifiers such as $\forall x_0\: \forall x_1\:\ldots$.

Furthermore, the Chang model is an inner model of the set theoretic universe $\VVV$. From the perspective of mathematical realism, this implies that any absolute proposition that we establish in the Chang model is \emph{true}, rather than merely consistent. Formally, such a proposition is provable from the existence of a proper class of Woodin cardinals that are limits of Woodin cardinals.

Finally, if there is a proper class of Woodin cardinals that are limits of Woodin cardinals, then the Chang model satisfies $\mathbf{AD^+}$ \cite{Larson}*{p. 44}. Although we do not use this axiom in the present paper, it is quite relevant to our approach. This axiom implies that the bounded real-valued functions on a complete separable metric space $X$, can be obtained from the continuous bounded real-valued functions on $X$ by iterating taking limits of monotone well-ordered sequences. By analogy with Pedersen's enveloping Borel $*$-algebra $\Bb(A)$ \cite{Pedersen72}, we can define the enveloping $\infty$-Borel $*$-algebra $\Bb_\infty(A)$ of a C*-algebra $A$. In general, $\Bb_\infty(A) \subsetof V^*(A)$, but it is possible to show that $\Bb_\infty(A) = V^*(A)$ whenever $A$ is type I. The reader is cautioned that the axioms $\mathbf{BP}$ and $\mathbf{DC}$ imply that the meager subsets of any Polish space are closed under well-ordered unions \cite{Kechris73}*{proof of proposition 1.5.1}, so the bounded complex-valued functions on $\TT/\QQ$ with meager support form an $\infty$-Borel $*$-algebra that is not a V*-algebra; it is an ideal of $\ell^\infty(\TT/\QQ)$ of codimension $1$, and the corresponding unital $*$-homomorphism $\ell^\infty(\TT/\QQ) \To \CC$ is $\infty$-normal, in the sense that it respects limits of monotone well-ordered sequences, but it is not continuum-weakly continuous.

\textbf{Hugh Woodin} explains that, assuming the existence of a proper class of Woodin cardinals that are limits of Woodin cardinals, it is possible to obtain a transitive model of $\mathbf{ZF}$ that is closed under countable unions and satisfies $\mathbf{DC + AC_\RR + AD^+}$, where $\mathbf{AC_\RR}$ denotes the axiom of choice for all families of sets indexed by $\RR$, an axiom that is strictly stronger than $\mathbf{AC_{ae}}$ \cite{Woodin14} \cite{Sargsyan09}. Our motivation for eschewing this convenient assumption is the observation that unnecessary choice principles can interfere with analytic structure, e. g., measure, classification, etc. For a model with such a simple, natural definition, the Chang model is remarkably well suited to our approach. Thus, in a spirit of mathematical realism, we proceed into this natural landscape, rather than retreating into a carefully constructed model for the sake of a choice principle that we do not yet need.

\section{Absoluteness}

Some familiar results fail in the Chang model, but many continue to hold. How can we verify that a familiar theorem continues to hold in the Chang model?

The brute force approach is to carefully check the proof to see whether it uses only a fragment of choice provable from $\mathbf{ZF + DC +AC_{ae}}$. It is necessary to check the proof down to the foundations, i. e., to check also the proofs of logically preceding results. It would take a good deal of time and care to scrutinize the material of the standard textbooks in some subject area. Yet, after this work is complete, we will have established only the most basic results of a single branch of mathematics. The task of scrutinizing a significant portion of published mathematical research in this way is effectively insurmountable.

The better approach is to appeal to the absoluteness of many mathematical properties. This approach demands the we scrutinize the \emph{statement} rather than the proof of a familiar mathematical theorem. The statement of Fermat's last theorem is very simple, and we can easily verify that it holds in the Chang model $\CCC$ without examining its proof at all.

A mathematical property is absolute if it is insensitive to the ambient set theory. Formally, a formula of the language of set theory is absolute in case it holds for a tuple of sets in $\CCC$ iff it holds for that tuple in $\VVV$, and a mathematical property is absolute in case it is expressed by an absolute formula in the standard interpretation of mathematics in set theory. This condition is more properly termed ``absolute for $\CCC$'', but the Chang model is the only inner model that we discuss, so the term ``absolute'' always implicitly refers to absoluteness for the Chang model.

A theorem is a provable property of the empty tuple. Thus, an absolute theorem is true in the Chang model, and moreover, any proposition that is established in the Chang model is also true in the full set theoretic universe if its statement is absolute. Arithmetic statements, such as Fermat's last theorem are of this kind. However, we will mostly verify theorems in $\CCC$ using absolute properties of nonempty tuples: if an absolute property holds for all sets in $\VVV$, then it holds for all sets in $\CCC$, simply because every member of $\CCC$ is a member of $\VVV$.

\subsection*{Summary}
We use the terms ``mathematical object'' and ``mathematical property'' as primitive notions, whose meaning is intuitively clear. The standard interpretation is a particular assignment of hereditary sets to mathematical objects, and of set theoretic formulas to mathematical properties, such that a given property holds for a tuple of objects iff the corresponding formula is true of the corresponding tuple of sets. The Chang model is a class of hereditary sets, and we say that an object is in the Chang model iff the corresponding hereditary set is in the Chang model. We also say that a property holds in the Chang model iff the corresponding formula is true in the Chang model. Finally, we say that a property is absolute in case each tuple of objects in the Chang model has that property in the Chang model iff it actually has that property.

\begin{principle}
Let $\tau$ be a theorem of ordinary mathematics, i. e., of $\mathbf{ZF +AC}$. If $\tau$ is of the form $$\forall x_1\: \forall x_2\: \cdots \forall x_n\: \pi(x_1, \ldots x_n)$$ for some \emph{absolute} property $\pi$, then $\tau$ holds in the Chang model $\CCC$.
\end{principle}

\section{Establishing absoluteness}

The core of this section is a list of absolute properties that the reader can use to verify theorems as they come up. This list is preceded by several clarifications to help orient the reader.

\begin{remark}
The absoluteness of a formula of the language of set theory is usually a straightforward consequence of the following basic principles, whose meaning we do not explain:
\begin{enumerate}[i.]
\item Absolute formulas are closed under Boolean combination and bounded quantification.
\item Being the initial segment $V_{\omega +1}$ of the set theoretic universe is absolute.
\item Being the set of countable sequences from a given set is absolute.
\item Being the image of a given set under an absolute functional class is absolute. 
\end{enumerate}
These principles are sufficient to obtain the provided list of absolute properties in the standard interpretation. Furthermore, a minority of the properties on this list is sufficient to obtain the rest of the list without reference to the standard interpretation, because complex mathematical objects such as C*-algebras are widely understood to be defined in terms of basic mathematical objects such as sets, functions, and numbers.
\end{remark}

\begin{remark}
For convenience, we rely on one convention of the standard interpretation: we assume that each structure, such as a group, a metric spaces, or an operator algebra, is the tuple of its parts, in the style of Bourbaki. For example a group is a $4$-tuple $(G, \,\cdot\,, e, \inv)$. This frees us from having to state separately that, for example, being the metric of a metric space is an absolute property.
\end{remark}

\begin{remark}
Mathematical assertions are often phrased in terms of definite descriptions such as ``the set of natural numbers'' or ``the spectrum of $a$''. We enable definite descriptions by extending the language of set theory to include terms of the form
$$\iota x\: \pi(x, y_0, \ldots, y_n),$$
which denotes the unique $x$ that satisfies $\pi(x,y_0, \ldots, y_n)$ for the tuple $(y_0, \ldots, y_n)$, provided there exists a unique such $x$ for all tuples $(y_0, \ldots, y_n)$. If there is a unique such $x$ in both the set theoretic universe $\VVV$, and in the Chang model $\CCC$, and furthermore the formula $\pi(x, y_0, \ldots, y_n)$ is absolute, then for all values of $(y_0, \ldots, y_n)$ in $\CCC$, the term $\iota x\: \pi(x, y_0, \ldots, y_n)$ denotes the same set in $\VVV$ as it denotes in $\CCC$; in this case we say that $\iota x\: \pi(x, y_0, \ldots, y_n)$ is absolute.

If the property $\ppp$ is expressed by the formula $\pi(x,y_0,\ldots,y_n)$, then the definite description ``the $x$ such that $\ppp$'' is expressed by the term $\iota x\: \pi(x, y_0, \ldots, y_n)$, and we say that this definite description is absolute in case the term $\iota x\: \pi(x, y_0, \ldots, y_n)$ is absolute. In practice, a definite description is only meaningful for certain types of parameters, e. g., ``the Euler characteristic of $X$'' is not defined for C*-algebras $X$. We therefore adopt the convention that definite descriptions denote the empty set for parameters that are not of the appropriate type. This convention is of no practical significance, as a careful use of definite descriptions is a basic part of natural mathematical thought.

We summarize the situation by saying that we may freely use definite descriptions, just as we do in ordinary mathematical discourse, and that an absolute definite description is one that denotes the same object in both $\CCC$ and $\VVV$. A definite description ``the $x$ such that $\ppp$'' is absolute if $\ppp$ is an absolute property that is a definition of $x$ in both $\VVV$ and $\CCC$, and the parameters of $\ppp$ range over an absolute class, i. e., over all tuples satisfying a given absolute property.

We will use the notation $[x \rightsquigarrow \ttt]$ to denote the substitution of the definite description $\ttt$ for the variable $x$.
\end{remark}

\begin{remark}
``The class of groups'' is not a definite description in the sense above. We can express individual classes by formulas, but we cannot encode all classes by hereditary sets; this would lead to Russell's paradox. We will instead say that ``the class of groups'' is a class description, and that a class description ``the class of $x$ such that $\ppp$'' is absolute in case $\ppp$ is absolute and every object $x$ satisfying $\ppp$ for appropriate parameters in the Chang model is itself in the Chang model. Note that if the class of objects $x$ such that $\ppp$ is always a set in both $\VVV$ and $\CCC$, then ``the class of $x$ such that $\ppp$'' is absolute iff ``the set of $x$ such that $\ppp$'' is absolute.

We treat category descriptions in the same way, but with some sensitivity to the category-theoretic viewpoint. A category description consists of properties that specify the class of objects, the class of morphisms, the composition of morphisms, etc. To be termed absolute, a category description should satisfy the following conditions: First, its constituent properties should be absolute. Second, for all appropriate parameters from the Chang model $\CCC$, every morphism between objects in $\CCC$ should itself be in $\CCC$. Third, for all appropriate parameters from $\CCC$, every object should be in $\CCC$ \emph{up to isomorphism}; we do not ask that these isomorphisms be canonical in any way. Thus the two categories need not be equal, but they must be weakly equivalent.  Note that the class description ``the class of $x$ such that $\ppp$'' is absolute iff the category description ``the category of $x$ such that $\ppp$ with only identity morphisms'' is absolute. Thus, definite descriptions of sets can be viewed as a special case of class descriptions, which, in turn, can be viewed as a special case of category descriptions.

We remark that, because the axiom of choice fails in the Chang model, when working with categories in the Chang model, the appropriate notion of equivalence of categories is anaequivalence \cite{Makkai96}.
\end{remark}

Below, $\ppp$ and $\qqq$ denote arbitrary absolute properties, $\ttt$ and $\sss$ denote arbitrary absolute definite descriptions, and $\nnn$ denotes an arbitrary numeral. All other symbols denote \emph{variables}, and symbols in parentheses denote \emph{bound variables} that do not name objects for which the stated property holds or fails. The following are absolute:

\begin{enumerate}
\item\label {A2}$\sss[x \rightsquigarrow \ttt]$
\item $\ppp[x \rightsquigarrow \ttt]$
\item not $\ppp$
\item $\ppp$ and $\qqq$ 
\item $\ppp$ or $\qqq$ 
\item if $\ppp$ then $\qqq$

\item $x$ is equal to $y$

\item $X$ is a set 
\item $x$ is an element of the set $X$
\item the set of elements ($x$) of the set $X$ such that $\ppp$
\item the set $X$ is a subset of the set $Y$
\item the union of sets $X$ and $Y$
\item the intersection of sets $X$ and $Y$
\item the set $X$ excluding the elements of the set $Y$
\item the union of the sets in the set $X$

\item the set $X$ is empty
\item there exists an element ($x$) in the set $X$ such that $\ppp$
\item for all elements ($x$) of the set $X$ it is the case that $\ppp$

\item $f$ is a function 
\item the domain of the function $f$
\item the codomain of the function $f$
\item\label{A1} the value of the function $f$ at an element $x$ of its domain
\item the function $f$ is injective
\item the function $f$ is surjective
\item the range of the function $f$
\item the surjective function that maps each element ($x$) of the set $X$ to $\ttt$
\item the inclusion function of the subset $X$ into the set $Y$
\item the identity function on the set $X$
\item the composition of composable functions $f$ and $g$
\item the inverse of invertible function $f$

\item the set of natural numbers
\item the natural number $\nnn$
\item the sum of natural numbers $n$ and $m$
\item the product of natural number $n$ and $m$
\item the natural number $n$ is less than the natural number $m$
\item the set of functions from the set of natural numbers to the set $X$

\item the set of $n$-tuples of elements of the set $X$ for the natural number $n$
\item the $m$-th element of the $n$-tuple $\overrightarrow{x}$ for the natural number $m$ less than the natural number $n$
\item the set of $n$-tuples whose $m$-th element for the natural number $m$ less than the natural number $n$ is an element of the $m$-th element of the $n$-tuple $\overrightarrow{X}$ of sets

\item the set $X$ is finite
\item the set $X$ is countable
\item the set of finite subsets of $X$
\item the set of countable subsets of $X$

\item\label{A3} the set of functions from the countable set $X$ to the set $Y$

\item $R$ is a binary relation
\item the domain of the binary relation $R$
\item the codomain of the binary relation $R$
\item the binary relation $R$ is symmetric
\item the binary relation $R$ is antisymmetric
\item the binary relation $R$ is transitive
\item the binary relation $R$ is reflexive
\item the binary relation $R$ is an equivalence relation
\item the set of equivalence classes of the equivalence relation $R$
\item the function taking each element of the domain of the equivalence relation $R$ to its equivalence class

\item $G$ is a group
\item $R$ is a ring
\item $F$ is a field
\item $A$ is an algebra
\item $V$ is a vector space
\item $X$ is a \underline{basic space}, i. e., $X$ is a set equipped with a topological basis
\item the function $f$ from the basic space $X$ to the basic space $Y$ is continuous
\item the closure of the subset $Y$ of the basic space $X$
\item $f$ is a net in the set $X$
\item the net $f$ in the basic space $X$ converges to the point $x$
\item the basic space $X$ is separable
\item the basic space $X$ is Polish, i. e., it is a set equipped with a separable completely metrizable topology
\item the set of continuous maps from the Polish space $X$ to Polish space $Y$
\item the set of open subsets of Polish space $X$
\item the set of closed subsets of Polish space $X$
\item the set of Borel subsets of Polish space $X$
\item the set of Borel functions from the Polish space $X$ to the Polish space $Y$

\item the ring of integers
\item the field of real numbers
\item the field of complex numbers

\item the function $f$ between Euclidean spaces is smooth
\item $M$ is a smooth manifold
\item the function $f$ between smooth manifolds is smooth

\item the extended real line
\item the Lebesgue integral as a function from Borel sets of real numbers to the extended real line
\item the set $X$ of real numbers is Lebesgue measurable
\item the extended complex plane
\item the function $f$ is holomorphic
\item the function $f$ is entire

\item $X$ is a metric space
\item the basic space obtained from the metric space $X$, i. e., the set of points of the metric space $X$ equipped with the topological basis of open balls
\item the metric space $X$ is complete
\item $\X$ is a Banach space
\item the dual space of the separable Banach space $\X$
\item $\X$ is a Banach algebra
\item $\H$ is a Hilbert space

\item the subset $X$ of the Hilbert space $\H$ is an orthonormal basis
\item $x$ is a bounded operator on the Hilbert space $\H$
\item $A$ is a concrete C*-algebra of operators on a Hilbert space $\H$
\item the space of ultraweakly continuous functionals on the concrete C*-algebra $A$

\item $A$ is an abstract C*-algebra
\item $\mu$ is a state on the abstract C*-algebra $A$
\item $\pi$ is a $*$-homomorphism from the abstract C*-algebra $A$ to the abstract C*-algebra $B$
\item $\rho$ is a $*$-homomorphic action of the abstract C*-algebra $A$ on the Hilbert space $\H$

\item $\H$ is a separable Hilbert space
\item the set of all separable concrete C*-algebras on the separable Hilbert space $\H$
\item the set of all von Neumann algebras on the separable Hilbert space $\H$

\item the state space of the separable C*-algebra $A$
\item the C*-algebra $A$ is approximately finite dimensional
\item the von Neumann algebra $M$ on a separable Hilbert space $\H$ is approximately finite dimensional

\item the separable C*-algebra $A$ is type I
\item the von Neumann algebra $M$ on the separable Hilbert space $\H$ is type $\mathrm{I}$
\item the von Neumann algebra $M$ on the separable Hilbert space $\H$ is type $\mathrm{II}$
\item the von Neumann algebra $M$ on the separable Hilbert space $\H$ is type $\mathrm{III}$

\item the factor $M$ on the separable Hilbert space $\H$ is type $\mathrm{I}_n$ for the natural number $n$
\item the factor $M$ on the separable Hilbert space $\H$ is type $\mathrm{I}_\infty$
\item the factor $M$ on the separable Hilbert space $\H$ is type $\mathrm{II}_1$
\item the factor $M$ on the separable Hilbert space $\H$ is type $\mathrm{II}_\infty$
\item the factor $M$ on the separable Hilbert space $\H$ is type $\mathrm{III}_\lambda$ for the real number $\lambda$

\item the category of countable graphs and graph morphisms
\item the category of countable groups and group homomorphisms
\item the category of second countable locally compact Hausdorff spaces and continuous maps
\item the category of second countable locally compact groups and continuous group homomorphisms
\item the category of second countable compact Hausdorff spaces and continuous maps
\item the category of complete separable metric spaces and contractive maps 
\item the category of Polish spaces and continuous maps
\item the category of standard Borel spaces and measurable functions
\item the category of separable Banach spaces and bounded linear maps
\item the category of separable Hilbert spaces and bounded linear maps 
\item the category of separable C*-algebras and bounded linear maps
\item the category of separable C*-algebras and $*$-homomorphisms
\item the category of separable C*-algebras and C*-morphisms
\item the category of von Neumann algebras on separable Hilbert spaces and ultraweakly continuous linear maps
\item the category of von Neumann algebras on separable Hilbert spaces and ultraweakly continuous unital $*$-homomorphisms
\item the category of Borel equivalence relations on Polish spaces and Borel reductions

\item $\alpha$ is an ordinal
\item the ordinal $\alpha$ is less than the ordinal $\beta$
\item the set of countable ordinals
\item the smallest uncountable ordinal
\end{enumerate}

\begin{example}[terms and equations]
We begin with item \ref{A1}: ``the value of the function $f$ at an element $x$ of its domain'' is an absolute definite description; symbolically, ``$f(x)$'' is absolute. Item \ref{A2} explains that substitution preserves absoluteness, so nested terms such as ``$f(g(h(x)))$'' are absolute. Next, we can deduce that being the $\mmm$-th element of a pair, or any $\nnn$-tuple is absolute, so ``the pair whose first element is $x_0$ and whose second element is $x_1$'' is an absolute definite description; symbolically ``$(x_0,x_1)$'' is absolute, and similarly for all $\nnn$-tuples. We now combine these observations to deduce that arbitrary terms formed of nested function symbols of arbitrary arity, e. g., ``$f(g(x,y),h(y))$'', are absolute. Finally, since the substitution of absolute definite descriptions into an absolute property yields another absolute property, we deduce that any equation is absolute.
\end{example}

\begin{example}[Associativity of addition]
The example above shows that the equation $(n+m)+k = n+(m+k)$ is absolute, so we might be tempted to jump to the conclusion that addition of natural numbers is associative in the Chang model, because it is associative normally, but we have said nothing about the variables $n$, $m$, $k$ and $+$! However, the property ``if $+$ is the addition of natural numbers and $n$ is a natural number and $m$ is a natural number and $k$ is a natural number, then $(n+m)+k = n+(m+k)$'' is absolute; since it holds for \emph{all} objects in $\VVV$, it holds for \emph{all} objects in $\CCC$. Thus, the associativity of addition is verified in the Chang model.
\end{example}

\begin{example}[quantification]
Consider the equation ``$x\cdot x \cdot x = 2$''. It is absolute, together with the specifications that $x$ denotes a real number, that $\cdot$ denotes muliplication, and that $2$ denotes the number two. This equation has a solution in the Chang model, just as it does normally, because ``there exists an element $x$ of the set of real numbers such that $x \cdot x \cdot x = 2$'' is absolute. We may similarly verify the solution of $x \cdot x \cdot x =2$ is unique in the Chang model, and thereby establish that the definite description ``the cube root of two'' is absolute.

It is crucial that we quantify over a set! The property ``$X$ is a nonmeasurable set of real numbers'' is absolute, but ``there exists an $X$ such that $X$ is a nonmeasurable set of real numbers'' is certainly not. The property ``there exists an element of the set of sets of real numbers $X$ such that $X$ is a nonmeasurable set of real numbers'' quantifies over the set of sets of real numbers, but it is not absolute because the definite description ``the set of sets of real numbers'' is not absolute. 
\end{example}

\begin{example}[Fermat's last theorem]
In order to show that the inequality ``$a^n + b^n \neq c^n$'', together with its usual specifications, is absolute, we need only to show that ``the exponentiation of natural numbers'' is an absolute definite description. The property ``$exp$ is a function, and the domain of $exp$ is the Cartesian square of the set of natural numbers, and the codomain of $exp$ is the set of natural numbers, and for all natural numbers $n$, $exp(n,0) = 1$, and for all natural numbers $n$ and $m$, $exp(n,m+1) = exp(n,m) \cdot n$'' is absolute, so we can verify that there is a unique such object in the Chang model by quantifying over ``the set of functions from the Cartesian square of the set of natural numbers to the set of natural numbers''; this is an absolute definite description by item \ref{A3}. Therefore, ``if $n$ is greater than two, then $a^n + b^n \neq c^n$'' is absolute, so we have verified that Fermat's last theorem holds in the Chang model.
\end{example}

\begin{example}[Fuglede's theorem: If $x$ and $y$ are bounded operators on a Hilbert space $\H$ and $y$ is normal, then $x y = yx$ implies that $x y^* = y^* x$]
We deduce that the following are absolute:
\begin{itemize}
\item $\H$ is a Hilbert space
\item $x$ and $y$ are bounded operators on $\H$
\item $y y^* = y^* y$
\item $xy = yx$ implies that $x y^* = y^*x$
\item if $x$ and $y$ are bounded operators on the Hilbert space $\H$, and $y$ is normal, then $x y = yx$ implies that $x y^* = y^* x$
\end{itemize}
\end{example}

\begin{example}[Kaplansky's density theorem: If $A\subsetof \B(\H)$ is a concrete C*-algebra, then the unit ball of $A$ is strongly dense in the unit ball of the strong closure of $A$] We deduce that the following are absolute:

\begin{itemize}
\item $A$ is a concrete C*-algebra on the Hilbert space $\H$
\item $x$ is a contraction on the Hilbert space $\H$ 
\item $a$ is in the unit ball of $A$
\item $\xi$ is in $\H$ and $\|(a-x) \xi\| \leq 1$
\item $T$ is a finite subset of $\H$ such that there exists an element $a$ of the concrete C*-algebra $A$ on the Hilbert space $\H$ such that for all $\xi$ in $T$ it is the case that $\|(a-x) \xi\| \leq 1$
\item $T$ is a finite subset of $\H$ such that there exists an element $a$ of the unit ball of the concrete C*-algebra $A$ on the Hilbert space $\H$ such that for all $\xi$ in $T$ it is the case that $\|(a-x) \xi\| \leq 1$
\item if $A$ is a concrete C*-algebra on the Hilbert space $\H$, and $x$ is a contraction on $\H$ with the property that every SOT neighborhood of $x$ contains an element of $A$, then every SOT neighborhood of $X$ contains an element of the unit ball of $A$
\end{itemize}
\end{example}

\begin{example}[Gelfand duality for separable commutative C*-algebras: If $A$ is a separable commutative C*-algebra then $A \iso C_0(\hat A)$]
We deduce that the following are absolute:
\begin{itemize}
\item $A$ is a separable commutative C*-algebra
\item the space of homomorphic states of the separable commutative C*-algebra $A$
\end{itemize}
We verify that the spectrum of a separable commutative C*-algebra is a Polish space. We deduce that the following are absolute:
\begin{itemize}
\item the set of continuous functions from the space of homomorphic states on the separable commutative C*-algebra $A$ to the set of complex numbers that vanish at infinity
\item if $A$ is a separable commutative C*-algebra, then $C_0(\hat A)$ is a C*-algebra 
\end{itemize}
We verify that if $A$ is a separable commutative C*-algebra, then $C_0(\hat A)$ is separable and complete. We deduce that the following are absolute:
\begin{itemize}
\item the set of $*$-isomorphisms from the separable commutative C*-algebra $A$ to $C_0(\hat A)$
\item the separable commutative C*-algebra $A$ is isomorphic to $C_0(\hat A)$
\end{itemize}
\end{example}

\begin{example}[If $V$ is a closed subspace of the Hilbert space $\H$, then there exists an orthogonal projection operator $p$ such that $p\H = V$]
We deduce that the following are absolute:
\begin{itemize}
\item $V$ is a closed subspace of the Hilbert space $\H$
\item there exists an element of the closed subspace $V$ of the Hilbert space $\H$ closest to the element $\xi$ of $\H$
\end{itemize}
We verify that if $V$ is a closed subspace of $\H$, and $\xi$ is an element of $\H$, then there exists a unique element of $V$ closest to $\xi$. We deduce that the following are absolute:
\begin{itemize}
\item the element of the closed subspace $V$ of the Hilbert space $\H$ closest to the element $\xi$ of $\H$
\item the function taking each element of the Hilbert space $\H$ to the closest element of the closed subspace $V$
\item the function taking each element of the Hilbert space $\H$ to the closest element of the closed subspace $V$ is an orthogonal projection operator on $\H$ whose image is $V$
\end{itemize}
We verify that if $V$ is a closed subspace of the Hilbert space $\H$, then the function $p$ taking each element of $\H$ to the closest element of $V$ is an orthogonal projection operator on $\H$ such that $p\H = V$; thus, such an orthogonal projection operator exists.
\end{example}

\section{Cheat sheet}\label{cheatsheet}

This section lists theorems in the theory of operator algebras that hold in the Chang model. A theorem that is verifiable by a straightforward absoluteness argument is terminated with a period. A theorem whose verification requires the scrutiny of a substantial part of its usual proof for applications of the axiom of choice is terminated with two periods.. A theorem whose verification requires a proof different from its usual proof is terminated with three periods... A theorem which fails ordinarily, but which holds in the Chang model is punctuated with a exclamation mark! Theorems punctuated in these last two ways are addressed in \cref{appendix1}.

The following theorems hold in the Chang model:

\medskip

\noindent Let $X$ be a metric space.
\begin{remark}\label{v1}
The metric space $X$ has a completion..
\end{remark}

\begin{remark}\label{v2}
If $X$ is complete, then the intersection of a countable family of dense open sets is dense.
\end{remark}

\noindent Let $X$ be a complete separable metric space.

\begin{remark}[cf. \cite{Kechris95} theorem 15.6]\label{v3}
If $X$ is uncountable, then there is a bijection from $\II$ to $X$ such that the preimage any Borel set is Borel, and the image of any Borel set is Borel.
\end{remark}

\begin{remark}[cf. \cite{Kechris95} theorem 17.41]\label{v3.1}
If $m$ is an atomless Borel probability measure on $X$, then there is a bijection from $\II$ to $X$ such that the preimage of any Borel set is Borel, the image of any Borel set is Borel, and the pushforward of Lebesgue measure on the unit interval is $m$.
\end{remark}

\begin{remark}\label{v3.2}
If $m$ is a Borel probability measure on $X$, then for every subset $S \subsetof X$, there are Borel subsets $B_0, B_1 \subsetof X$ such that $B_0 \subsetof S \subsetof B_1$ and $m(B_0) = m(B_1)$!
\end{remark}

\noindent Let $m$, $m_0$, and $m_1$ be (totally defined) probability measures on sets $T$, $T_0$, and $T_1$, respectively, with each set injectable into $\RR$.

\begin{remark}\label{v3.4}
The measure $m$ is a pushforward of Lebesgue measure on $\II$, i. e., there is a function $f\: \II \To T$ such that $m(X)$ is equal to the Lebesgue measure of $f\inv(X)$ for each $X \subsetof T$!
\end{remark}

\begin{remark}\label{v3.5}
If $f \: T_0 \times T_1 \To \CC$ is a function such that
\begin{enumerate}
\item $\int_{t_0 \in T_0} \int_{t_1 \in T_1} |f(t_0, t_1)|\, dm_1\, dm_0 < \infty $,
\item $\int_{(t_0, t_1) \in T_0 \times T_1} |f(t_0, t_1)|\, d(m_0 \times m_1) < \infty$, or
\item $\int_{t_1 \in T_1} \int_{t_0 \in T_0} |f(t_0, t_1)|\, dm_0 \, dm_1< \infty$,
\end{enumerate}
then
\begin{align*}
\int_{t_0 \in T_0} \int_{t_1 \in T_1} f(t_0, t_1)\, dm_1\, dm_0 & = \int_{(t_0, t_1) \in T_0 \times T_1} f(t_0, t_1)\, d(m_0 \times m_1)
\\ &= \int_{t_1 \in T_1} \int_{t_0 \in T_0} f(t_0, t_1)\, dm_0 \, dm_1!
\end{align*}
\end{remark}

\noindent Let $X$ be a topological space, let $Y \subsetof X$ be a subspace, and let $f\: X \To Z$ be a function to another topological space.

\begin{remark}[cf. \cite{Pedersen88} proposition 1.3.6]\label{v3.73}
A point belongs to the closure of the set $Y$ iff there is a net in $Y$ converging to that point...
\end{remark}

\begin{remark}[cf. \cite{Pedersen88} proposition 1.4.3]\label{v3.75}
The function $f$ is continuous at $x \in X$ iff for every net $(x_\lambda)$ converging to $X$, the net $(f(x_\lambda))$ converges to $f(x)$... 
\end{remark}

\noindent Let $\X$ and $\Y$ be Banach spaces.

\begin{remark}\label{v4}
Any linear function from $\X$ to $\Y$ is bounded!
\end{remark}

\begin{remark}\label{v5}
A surjective linear function from $\X$ to $\Y$ is open!
\end{remark}

\begin{remark}\label{v6}
If $\{T_\lambda\}$ is a family of (bounded) linear functions from $\X$ to $\Y$ such that $\{T_\lambda x\}$ is bounded for all $x \in X$, then $\{\|T_\lambda\|\}$ is bounded. 
\end{remark}

\noindent Let $\X$ be a separable Banach space.

\begin{remark}\label{v7}
Every bounded complex-valued linear function defined on a subspace of $\X$ extends to a linear function of the same norm on all of $\X$.
\end{remark}

\begin{remark}\label{v8}
The unit ball of $\X^*$ is Polish and compact, in the weak* topology.
\end{remark}

\begin{remark}\label{v9}
If $K \subsetof \X^*$ is bounded and weak*-closed, then every element of $K$ is in the weak*-closed convex hull of the extreme points of $K$, and is the barycenter of a Borel probability measure concentrated on the extreme points of $K$.
\end{remark}

\noindent Let $\H$ be a Hilbert space.

\begin{remark}
If $\varphi$ is a functional on $\H$, then there exists a vector $\xi \in \H$ such that $\varphi(\eta) = \< \xi | \eta\>$ for all $\eta \in \H$.
\end{remark}

\begin{remark}
If $\psi$ is a bounded sesquilinear form on $\H$, then there exists a bounded operator $x$ on $\H$ such that $\psi(\xi, \eta) = \< \xi | x \eta\>$ for all $\xi, \eta \in \H$.

\end{remark}

\begin{remark}
The function $p \mapsto p\H$ is a bijection between projection operators on $\H$ and its closed subspaces.
\end{remark}

\begin{remark}Every monotonically decreasing net of positive operators on $\H$ has a greatest lower bound, and converges to it ultrastrongly.
\end{remark}

\noindent Let $A \subsetof \B(\H)$ be a nondegenerate concrete C*-algebra.

\begin{remark}\label{v20}
If $\varphi$ is a vector functional on $A$, then there exist $\xi, \eta \in \H$ such that $\|\xi\|^2 = \|\varphi \| = \|\eta\|^2$ and $\varphi(a) = \<\eta| a \xi\>$ for all $a \in A$.
\end{remark}

\begin{remark}\label{v14}
If $\varphi$ is a vector functional on $A$, then there exist unique positive functionals $\varphi_+$ and $\varphi_-$ on $A$ such that $\varphi = \varphi_+ - \varphi_-$ and $\|\varphi\| = \|\varphi_+\| + \|\varphi_-\|$.
\end{remark}

\begin{remark}\label{v14.1}
If $\varphi$ is a vector functional on $A$, then there exists a unique positive functional $|\varphi|$ on $A$ such that $\|\,|\varphi|\,\| = \|\varphi\|$ and $|\varphi(a)|^2 \leq \|\varphi\| |\varphi|(a^*a)$ for all $a \in A$.
\end{remark}

\begin{remark}\label{v21}
The ultraweak closure of $A$ is equal to its double commutant $A''$..
\end{remark}

\begin{remark}\label{v22}
If $x$ is in the ultraweak closure of $A$, and $\|x\| \leq 1$, then $x$ is in the ultraweak closure of the unit ball of $A$.
\end{remark}

\begin{remark}[cf. \cite{Pedersen79} theorem 2.4.3]\label{v23}
If $\H$ is separable, then $A$ is a von Neumann algebra iff it closed under limits of ascending sequences.
\end{remark}

\begin{remark}[cf. \cite{Pedersen79} section 3.12]\label{v24}
The centralizers of $A$ form a C*-algebra that is isomorphic to the C*-algebra of multipliers of $A$.
\end{remark}

\noindent Let $B \subsetof \B(\H)$ be a concrete C*-algebra that is closed under limits of ascending sequences, e. g., a von Neumann algebra; see \cite{Pedersen79}*{section 4.5}.

\begin{remark}\label{v25}
The support projection of every self-adjoint operator in $B$ is itself in $B$.
\end{remark}

\begin{remark}\label{v26}
The projections of $B$ are closed under countable meets and joins.
\end{remark}

\begin{remark}\label{v27}
For every operator $b \in B$, there exists a unique partial isometry $u$ such that $u^*u$ is the support projection of $|x| = (x^* x)^{\frac 1 2}$ and $x = u |x|$.
\end{remark}

\noindent Let $A$ be an abstract C*-algebra.

\begin{remark}\label{v10}
The positive elements of $A$ of norm strictly less than $1$ form an approximate unit. 
\end{remark}

\begin{remark}\label{v12}
Every normal element of $A$ has the continuous functional calculus.
\end{remark}

\begin{remark}[GNS]\label{v13}
For each state $\mu$ on $A$, there exists a representation $\gamma_\mu\: A \To \B(\H_\mu)$ and a cyclic vector $\xi_\mu$ such that $\< \xi_\mu | \gamma_\mu(a) \xi_\mu \> = \mu(a)$ for all $a \in A$..
\end{remark}

\begin{remark}\label{v.13.1}
If $\pi\: A \To \B(\H)$ is a representation with cyclic vector $\eta_0$, $\mu = \<\eta_0| \pi(\cdot) \eta_0\>$, and $\gamma_\mu\: A \To \B(\H_\mu)$ is the GNS representation for $\mu$, then there exists a unique unitary operator $u$ from $\H$ to $\H_\mu$ such that $u \eta_0 = \xi_\mu$ and $u \pi(a) = \gamma_\mu(a) u$ for all $a \in A$..
\end{remark}

\begin{remark}\label{v13.2}
For each state $\mu$ on $A$, the GNS representation $\gamma_\mu\: A \To \B(\H_\mu)$ is irreducible iff $\mu$ is pure..
\end{remark}

\begin{remark}\label{v15}
The ultraweak closure of $A$ in its universal representation is an enveloping von Neumann algebra of $A$, i. e., every $*$-homomorphism from $A$ into a von Neumann algebra factors uniquely through this ultraweak closure via an ultraweakly continuous $*$-homomorphism...
\end{remark}

\noindent Let $A$ be a separable C*-algebra.

\begin{remark}\label{v16}
The universal representation of $A$ is faithful.
\end{remark}

\begin{remark}\label{v17}
If $A$ is commutative, then the spectrum $\hat A$ is a locally compact Polish space such that $A \iso C_0(\hat A)$.
\end{remark}

\begin{remark}\label{v18}
The convex hull of the pure states of $A$, together with $0$, is weak* dense in the quasi-state space of $A$. Every state of $A$ is the barycenter of a Borel probability measure on its pure state space.
\end{remark}

\begin{remark}[cf. \cite{Pedersen79} theorem 6.8.7]\label{v19}
The following are equivalent:
\begin{enumerate}[\qquad(i)]
\item $A$ is a C*-algebra of type I,
\item $\Bb(A)$ is a Borel $*$-algebra of type I.
\setcounter{enumi}{3}
\item $A$ has a composition series in which each quotient has continuous trace.
\item The image of every irreducible representation of $A$ includes the compact operators.
\item Two irreducible representations of $A$ are unitarily equivalent iff they have the same kernel.
\item The Borel structure on $\hat A$ generated by the Jacobson topology is standard.
\item Pedersen's Davies Borel structure on $\hat A$ is countably separated.
\item The Mackey Borel structure on $\hat A$ is countably separated.
\item Every factor representation of $A$ is type I.
\item $A$ has no factor representations of type II.
\item $A$ has no factor representations of type III.
\end{enumerate}
\end{remark}

\begin{remark}[cf. \cite{Pedersen79} proposition 6.3.2]\label{v19.5} If $A$ is type I, then the Mackey Borel structure and Pedersen's Davies Borel structure coincide with the Borel structure on $\hat A$ generated by the Jacobson topology.
\end{remark}

\noindent Let $M, N \subsetof \B(\H)$ be  von Neumann algebras on a separable Hilbert space.

\begin{remark}\label{v28}
Every functional on $M$ is ultraweakly continuous!
\end{remark}

\begin{remark}\label{v30}
If $M$ and $N$ are approximately finite dimensional factors, both of type $\mathrm{I}_n$ for some $n \in \NN \cup \{\infty\}$, of type $\mathrm{II}_k$ for some $k \in \{1, \infty\}$, or of type $\mathrm{III}_\lambda$ for some $\lambda \in (0,1]$, then $M \iso N$.
\end{remark}

\section{Basic definitions}\label{part2}

\begin{remark}\label{1}
A \underline{continuum} in a set $X$ is a function $\II \To X$.
\end{remark}

\begin{definition}\label{2}
Let $\H$ be a Hilbert space. The \underline{continuum-weak topology} on $\B(\H)$ is given by functionals of the form 
\begin{align}\tag{I} x \mapsto \int_0^1 \langle \eta_t | x\xi_t \rangle \,dt
\end{align}
for families $(\eta_t \in \H)$ and $(\xi_t \in \H)$ such that the functions $(\|\eta_t\|^2\: t \in \II)$ and $(\|\xi_t\|^2\:t \in \II)$ are integrable with respect to Lebesgue measure.
\end{definition}

\begin{remark}
When working with bounded operators, we will always use the closure line $\overline{(\cdot)}$ to denote closure in the continuum-weak topology.
\end{remark}

\begin{remark}\label{3}
Every probability measure on a set $T \preccurlyeq \RR$ is a pushforward of Lebesgue measure on the unit interval; see \ref{v3.4}. It follows that whenever $(\eta_t)$ and $(\xi_t)$ are families of vectors such that the functions $(\|\eta_t\|^2)$ and $(\|\xi_t\|^2)$ are integrable with respect to a probability measure $m$ on some set $T \preccurlyeq \RR$, the functional $x \mapsto \int_{t \in T} \< \eta_t | x \xi_t\> \,dm$ is of the form (I). Clearly, the same is also true of any finite measure $m$ on a set $T \preccurlyeq \RR$. It follows that functionals of the form (I) are closed under addition and scalar multiplication, so \emph{every} continuum-weakly continuous functional is of the form (I).
\end{remark}

\begin{remark}\label{4}
The continuum-weak topology is stronger than the ultraweak topology, but weaker than the norm topology. 
\end{remark}

\begin{remark}\label{5}
The adjoint operation is continuum-weakly continuous, operator addition is jointly continuum-weakly continuous, and operator multiplication is continuum-weakly continuous in each variable. 
\end{remark}

\begin{definition}\label{6}
A concrete C*-algebra $E \subsetof \B(\H)$ is a \underline{V*-algebra} in case it is closed in the continuum-weak topology.
\end{definition}

\begin{remark}\label{7}
Every von Neumann algebra is a V*-algebra.
\end{remark}

\section{Continuum amplification}

\begin{remark}\label{8}
Let $E \subsetof \B(\H)$ be a V*-algebra. Clearly, a functional $\varphi: E \To \CC$ is continuum-weakly continuous iff it is a vector functional in the canonical representation of $E$ on the Hilbert space $L^2(\II, \H)$. Note that, in general, the isometry $u \: L^2(\II) \tensor \H \To L^2(\II, \H)$ given by $u(f \tensor \xi)(t) = f(t) \xi$ is not unitary.
\end{remark}

\begin{proposition}\label{9}
If $(\varphi_s)$ is an indexed family of continuum-weakly continuous functionals on $E$ such that $(\|\varphi_s\|\: s \in \II)$ is integrable, then the functional $\varphi: x \mapsto \int_0^1 \varphi_s(x) \, ds$ is also continuum-weakly continuous. 

\end{proposition}

\begin{proof}
Each continuum-weakly continuous functional is a vector functional for the canonical representation $\pi: E \To \B(L^2(\II, \H))$. Therefore, we can apply $\mathbf{AC_{ae}}$ to choose, for almost all $s \in \II$, vectors $\xi^s = [\xi^s_t \:t \in \II]$ and  $\eta^s = [\eta^s_t \:t \in \II]$ in  $L^2(\II, \H)$, such that $\|\xi^s\|^2 = \|\varphi_s\| = \|\eta^s\|^2$, and $\varphi_s: x \mapsto \langle \eta^s |\pi(x) \xi^s \rangle$, i. e.,
$$ 
\int_0^1 \| \xi^s_t\|^2 \, dt  = \|\varphi_s\| = \int_0^1   \|\eta^s_t \|^2 \, dt
$$
and
$$ \varphi_s(x) = \int_0^1 \langle \eta^s_t | x \xi^s_t \rangle \, dt.$$
Applying Tonelli's theorem, we find that the function $(s,t) \mapsto \| \xi^s_t\|^2$ is integrable on $\II \times \II$, as is $(s,t) \mapsto \|\eta^s_t\|^2$, so by Fubini's theorem,
$$\int_0^1 \varphi_s(x) \, ds = \int_0^1\left( \int_0^1 \langle \eta^s_t | x\xi^s_t \rangle \, dt \right) \, ds =\int_{(s,t) \in \II \times \II} \langle \eta^s_t | x\xi^s_t \rangle \, d(s,t) $$
for all $x \in E$. Thus, $\varphi\: x \mapsto \int_0^1 \varphi_s(x)\, ds$ is continuum-weakly continuous by \cref{3}.
\end{proof}

\begin{lemma}\label{10}
Let $(\varphi_n)$ be a norm-convergent sequence of functionals on $E$ converging to $\varphi$. If $\varphi_n$ is continuum-weakly continuous for each $n$, then so is $\varphi$.
\end{lemma}

\begin{proof}
Without loss of generality, we can assume that for all $n$, $\|\varphi_{n+1} - \varphi_n \| \leq 2^{-n}$. Writing $\psi_n = \varphi_{n+1} - \varphi_n$, we have that $\|\psi_n \| \leq 2^{-n}$, so $\sum_n \|\psi_n \| \leq 2$. Applying \cref{9}, and the fact that, by \ref{v3.4}, every probability measure on a countable set is a pushforward of Lebesgue measure on $\II$, we find that $\varphi = \sum_n \psi_n$ is continuum-weakly continuous. 
\end{proof}

\begin{lemma}\label{11}
Each self-adjoint continuum-weakly continuous functional $\varphi$ on $E$ has a Jordan decomposition $\varphi= \varphi_+ - \varphi_-$, where $\varphi_+$ and $\varphi_-$ are positive continuum-weakly continuous functionals, and $\|\varphi\| = \|\varphi_+\| + \|\varphi_-\|$. Thus, every continuum-weakly continuous function on $E$ is a linear combination of continuum-weakly continuous states.
\end{lemma}

\begin{proof}
Each self-adjoint continuum-weakly continuous functional $\varphi$ is a vector functional for the canonical representation of $E$ on $L^2(\II, \H)$. Every self-adjoint vector functional on a concrete C*-algebra has a Jordan decomposition into vector functionals (\ref{v14}).
\end{proof}

\begin{lemma}\label{12}
Let $\mu$ be a continuum-weakly continuous state on $E$. There exists a family $(\xi_t) \in \L^2(\II, \H)$ such that $\mu: x \mapsto \int_0^1 \< \xi_t| x \xi_t\> dt$, and $\int_0^1 \|\xi_t\|^2 \,dt= 1$
\end{lemma}

\begin{proof}
The state $\mu$ is a vector functional for the canonical representation of $E$ on $L^2(\II, \H)$.
\end{proof}

\begin{lemma}\label{13}
Let $E \subsetof \B(\H)$ and $F \subsetof \B(\K)$ be V*-algebras. A (bounded) linear function $\pi: E \To F$ is continuum-weakly continuous iff the pullback of every vector state is continuum-weakly continuous.
\end{lemma}

\begin{proof}
The forward direction is trivial. Therefore, it remains to show that if the pullback of every vector state is continuum-weakly continuous, then the pullback of every continuum-weakly continuous functional is continuum-weakly continuous. Each such functional is a linear combination of continuum-weakly continuous states, and each such state is an integral of vector states, so \cref{9} is sufficient to establish the claim.
\end{proof}

\begin{lemma}\label{13.5}
Let $M\subsetof \B(\H)$ and $\N \subsetof \B(\K)$ be von Neumann algebras. An ultraweakly continuous linear function $\pi\: M \To N$ is continuum-weakly continuous.
\end{lemma}

\begin{proof}
This is a corollary of the preceding lemma.
\end{proof}

\begin{proposition}\label{14}
Let $E\subsetof \B(\H)$ be a V*-algebra, and let $\rho:E \To \B(L^2(\II,\H))$ be its canonical representation on $L^2(\II, \H)$. Then, $\rho$ is a continuum-weakly homeomorphic $*$-isomorphism of $E$ onto $\rho(E)$, which is itself a V*-algebra. 
\end{proposition}

\begin{proof}
Without loss of generality, assume that $E = \B(\H)$. Fix $f \in L^2(\II)$, and let $u_f: \H \To L^2(\II, \H)$ be the isometry defined by $u_f\xi = [f(t) \xi\: t \in \II]$. Thus, $\pi: x \mapsto u_f^* x u_f$ is a continuum-weakly continuous map $\B(\II, \H) \To \B(\H)$ such that $\pi \circ \rho$ is the identity on $\B(\H)$. The canonical representation $\rho$ is itself continuum-weakly continuous because the pullback of every vector functional is trivially continuum-weakly continuous. The rest of the proof is elementary general topology.
\end{proof}

\section{Projections in a V*-algebra}

\begin{lemma}\label{15}
Let $\H$ be a Hilbert space. If $(x_n)$ is a sequence in $\B(\H)$ that converges to $x$ in the ultraweak topology, then it converges to $x$ in the continuum-weak topology.
\end{lemma}

\begin{proof}
The set $\{\|x_n\|\}$ is bounded by some positive real number $C$, by an application of the uniform boundedness principle, since $\B(\H)$ is isometrically isomorphic to the dual of $\B^1(\H)$, the Banach space of trace class operators on $\H$. It follows that for all families $(\xi_t \in \H\: t \in \II)$ and $(\eta_t \in \H\:t \in \II)$ in $\L^2(\II, \H)$,
$$ \int \langle \eta_t |x_n \xi_t \rangle \, dt \To \int \langle \eta_t | x \xi_t \rangle \, dt,$$
by an application of the dominated convergence theorem, since $|\langle \eta_t| x_n \xi_t \rangle| \leq C \cdot \|\xi_t\| \cdot \| \eta_t\|$.
\end{proof}

\begin{lemma}\label{31}
Let $E$ and $F$ be V*-algebras, and let $\psi: E \To F$ be a continuum-weakly continuous positive map. Then, $\psi$ is sequentially normal, in the sense that if $(x_n \in E\: n \in \NN)$ is a descending sequence of positive operators in $E$ converging ultraweakly to $0$, then $(\psi(x_n))$ is a descending sequence of positive operators in $F$ converging ultraweakly to $0$.
\end{lemma}

\begin{proof}
The various modes of convergence coincide for monotone sequences of positive operators. In particular, if the greatest lower bound of $(x_n)$ is $0$, then $x_n \To 0$ ultraweakly, and therefore continuum-weakly, by \cref{15}. By assumption, it follows that $\psi(x_n) \To \psi(0) =0$ continuum-weakly, and therefore ultraweakly.
\end{proof}

\begin{lemma}\label{16}
Let $E \subsetof \B(\H)$ be a V*-algebra. If $x \in E$ is positive, then its support projection $[x]$ is also in $E$.
\end{lemma}

\begin{proof}
$$[x] = \lim^{uw}_{n \To \infty} x^{\frac 1 n}$$
\end{proof}

\begin{proposition}\label{17}
The projections of $E$ are closed under countable meets and joins.
\end{proposition}

\begin{proof}
The projections of $E$ are closed under binary meets because
$$p \wedge q = \lim_{n \To \infty}^{uw} (pq)^n.$$ They are closed under binary joins because $p \vee q = [p+q]$. It follows that the projections of $E$ are closed under countably infinite meets and joins because $E$ is closed under limits of ascending and descending sequences.
\end{proof}

\begin{lemma}\label{18}
The projections of $E$ are an approximate unit for $E$. 
\end{lemma}

\begin{proof}
Recall that for any C*-algebra, its positive elements of norm strictly less than $1$ form an approximate unit. It follows that in $E$, the positive elements $r$ satisfying $r^2 = \alpha r$ for $\alpha \in (0,1)$ form an approximate unit. It is then straightforward to show that the projections themselves form an approximate unit.  
\end{proof}

\begin{lemma}\label{19}
If $E$ is nondegenerate, then for all $\xi \in \H$, there is a projection $p \in E$ such that $p \xi = \xi$.
\end{lemma}

\begin{proof}
By \cref{18} above, the projections of $E$ converge to the identity on $\H$ in the strong operator topology, and in particular $\lim_{p \nearrow 1} \|\xi   - p \xi\| = 0$. Using the axiom of dependent choices, we can obtain an increasing sequence $(p_n \: n \in \NN)$ of projections in $E$ such that $\lim_{n \To \infty} \|\xi- p_n \xi\| = 0 $. The ascending sequence $(p_n)$ converges ultraweakly to some projection $p$, which is therefore in $E$. Therefore, $\lim_{n \To \infty} \|p\xi  -  p_n\xi\| = 0$. We conclude that that $p \xi = \xi$.
\end{proof}

\begin{proposition}\label{20}
Let $E$ be a V*-algebra, and let $\varphi: E \To \CC$ be a continuum-weakly continuous functional. There exist projections $p, q \in E$ such that $\varphi(p x q) = \varphi(x)$ for all $x \in E$.
\end{proposition}

\begin{proof}
By \cref{14}, we may assume that $\varphi$ is a vector functional, so the existence of $p$ and $q$ then follow by \cref{19}, above.
\end{proof}

\begin{proposition}\label{21}
Let $X$ be a set. Then
$$F = \{f \in \ell^\infty(X) \: \supp(f) \not \succcurlyeq \II \}$$
is a V*-algebra in its canonical representation on $\ell^2(X)$.
\end{proposition}

\begin{proof}
It is easy to see that $F$ is a $*$-algebra. The continuum-weak closure of $F$ is a V*-algebra $\overline F\subsetof \ell^\infty(X)$. Suppose that $\overline F \neq F$. It follows that there is a positive function $f_\infty \in \overline F \setminus F$. By definition of $F$, $\supp(f_\infty) \succcurlyeq \II$, i. e., there is an injection $i\: \II \To \supp(f_\infty)$. We now define a continuum-weakly continuous functional on $\ell^\infty(X)$ by
$$\varphi\: h \mapsto \int_0^1 \<e_{i(t)}|he_{i(t)}\>\, dt = \int_0^1 h(t) \, dt.$$
Our choice of $i$ guarantees that $\varphi(f_\infty) >0$. However, for all $f \in F$, we have that $i(\II) \cap \supp(f) \not \succcurlyeq \II$, but $i(\II) \cap \supp(f) \preccurlyeq \II$ because $i$ is an injection, so by $\mathbf{PSP}$, $i(\II) \cap \supp(f) \preccurlyeq \NN$. It follows that $\varphi(f) = 0$ for all $f \in F$, contradicting our assumption that $f_\infty \in \overline F$.
\end{proof}

\begin{example}\label{22}
In particular, $\{f \in \ell^\infty(\RR)\: \supp(f) \preccurlyeq \NN\}$ is a V*-algebra on $\ell^2(\RR)$, so a V*-algebra need not be unital.
\end{example}

\section{Separable Hilbert spaces}

\begin{remark}\label{23}
Let $\H$ be a separable Hilbert space. Every continuum-weakly continuous functional on $\B(\H)$ is ultraweakly continuous, by \ref{v28} and \cref{11}, so the ultraweak topology and the continuum-weak topology coincide on $\B(\H)$.
\end{remark}

\begin{proposition}\label{24}
Every nondegenerate V*-algebra on a separable Hilbert space is a von Neumann algebra.
\end{proposition}

\begin{proof}
This is a corollary of \cref{23}, above.
\end{proof}

\begin{remark}
We will need the disintegration of ultraweakly continuous states on a direct integral of von Neumann algebras on separable Hilbert spaces. Rather than asking the reader to review Borel fields of Hilbert spaces, we will reprove these results, in part, in order to demonstrate the simplicity of the direct integral in the Chang model. We will work with a probability measure $m$ on an index set $T \preccurlyeq \RR$. 
\end{remark}

\begin{remark}\label{101}
Let $(\H_t \: t \in T)$ be a family of separable infinite-dimensional Hilbert spaces. We define the direct integral Hilbert space $\int^\oplus \H_t \, dm$ to consist of equivalence classes of square-integrable functions $(\xi_t \in \H_t\: t \in T)$. Applying $\mathbf{AC_{ae}}$, we can choose isomorphisms $\H_t \iso \ell^2(\NN)$ for almost every $t \in T$, so $\int^\oplus \H_t \, dm \iso L^2(T, \ell^2(\NN)) \iso \ell^2(\NN)$. This shows that $\int^\oplus \H_t \, dm$ is separable and complete.
\end{remark}

\begin{remark}
Let $(M_t \subsetof \B(\H_t) \: t \in T)$ be a family of von Neumann algebras. We define the direct integral von Neumann algebra $\int^\oplus M_t \, dm$ to consists of equivalence classes of essentially bounded functions $(x_t \in M_t\: t \in T)$. It is straightforward to show that the canonical action of $\int^\oplus M_t \, dm$ on the direct integral Hilbert space $\int^\oplus \H_t \, dm$ is well-defined and faithful. It is then straightforward to show that $\int^\oplus M_t \, dm$ is closed under limits of monotone countable sequences, and is therefore a von Neumann algebra; see \ref{v23}.
\end{remark}

\begin{lemma}\label{102}
Let $\varphi \: \int^\oplus M_t \, dm \To \CC$ be an ultraweakly continuous functional. There exists an integrable family of ultraweakly continuous functionals $(\varphi_t \in (M_t)_*\: t \in T)$ such that $$\varphi: [x_t \in M_t \:t \in T] \mapsto \int_{t \in T} \varphi_t(x_t)\, dm.$$
\end{lemma}

\begin{proof}
Since $\varphi$ is ultraweakly continuous, there exist two square-summable sequences of vectors in $\int_t \H_t \,dm$, which may be written as $([\xi^n_t\:t\in \H]\: n \in \NN)$ and $([\eta^n_t\:t\in \H]\: n \in \NN)$, such that
$$\varphi: [x_t \: t \in T] \mapsto \sum_n \int_t \< \xi^n_t| x_t \eta^n_t\> \, dm$$
The square-summability of the two sequences means that
$ \sum_n \int_t \| \xi^n_t\|^2 \, dm < \infty$ and $ \sum_n \int_t \| \eta^n_t\|^2 \, dm < \infty$. It follows by Tonelli's theorem that $(\xi^n_t\:n \in \NN)$ and $(\eta^n_t\: n \in \NN)$ are square-summable for almost every $t \in T$. Furthermore, since the elements of $\int^\oplus M_t dm$ are essentially bounded by definition, Fubini-Tonelli implies that
$$ \varphi: [x_t\: t \in T] \mapsto \int_t \sum_n \< \xi^n_t| x_t \eta^n_t\> \, dm.$$
\end{proof}

\begin{lemma}\label{39}
Every continuum-weakly continuous functional $\varphi$ on the von Neumann algebra $\bigoplus_t M_t$ is of the form $$(x_t:t \in T) \mapsto \int_t \varphi_t(x_t)\,dm$$
for some integrable family $(\varphi_t \in (M_t)_*\:t \in T)$ of ultraweakly continuous functionals, and some probability measure $m$ on $T$.
\end{lemma}

\begin{proof}
The functional $\varphi$ is a vector functional for the continuum amplification of the V*-algebra $\bigoplus_t M_t$ (\cref{14}), so there exists a unique positive functional $|\varphi|$ such that $\|\,|\varphi|\,\| = \| \varphi\|$, and $|\varphi(x)|^2 \leq \|\varphi\| \cdot |\varphi|(x^*x)$ for all $x \in \bigoplus_t M_t$ (\ref{v14.1}). The restriction of $|\varphi|$ to $\ell^\infty(T) \subsetof \bigoplus_t M_t$ satisfies $|\varphi|(\sum_n p_n) = \sum_n |\varphi|(p_n)$ for all countable families $(p_n)$ of pairwise disjoint projections in $\ell^\infty(T)$, for otherwise we would obtain a state on $\ell^\infty(\NN)$ that is not ultraweakly continuous. Thus, writing $p_A$ for the projection in $\ell^\infty(T)$ corresponding to a subset $A \subsetof T$, we find that $m \: A \mapsto|\varphi|(p_A)\cdot \| \varphi\|\inv $ is a probability measure on $T$. It is straightforward to show that if $(x_t) \in \bigoplus_t M_t$ vanishes almost everywhere, then $\varphi\: (x_t) \mapsto 0$; thus, $\varphi$ factors through the canonical map $(x_t) \mapsto [x_t]$:
$$
\begin{tikzcd}
\bigoplus_t M_t \arrow{r} \arrow{rd}{\varphi}
&
\int_t M_t\, dm  \arrow[dotted]{d}{\varphi'}\\
&
\CC
\end{tikzcd}
$$
Since $\int_t M_t\, dm$ is a von Neumann algebra on the separable Hilbert space $\int_t \H_t \, dm$, the functional $\varphi'$ is automatically (\ref{v28}) ultraweakly continuous, and \cref{102} yields the desired family $(\varphi_t \in (M_t)_*\: t \in T)$.
\end{proof}

\section{Canonical continuum predual}

\begin{remark}\label{25}
If $M \subsetof \B(\H)$ is a von Neumann algebra, then the canonical predual $M_*$ of $M$ consists of all ultraweakly continuous functionals $M \To \CC$. It is a predual of $M$ in the sense that $(M_*)^*$ is canonically isomorphic to $M$, but I don't know if it is unique up to isometric isomorphism. Of course, the predual is unique if $\H$ is separable.
\end{remark}

\begin{definition}\label{26}
If $E \subsetof \B(\H)$ is a V*-algebra, then the \underline{canonical continuum predual} $E_\bullet$ of $E$ consists of all continuum-weakly continuous functionals $E \To \CC$.
\end{definition}

\begin{proposition}\label{27}
Let $E \subsetof \B(\H)$ be a V*-algebra. The function $\iota\: E \To (E_\bullet)^*$ defined by $x \mapsto ( \varphi(x)\: \varphi \in E_\bullet)$ is an isometry onto the continuum-linear functionals, i. e., functionals $f \in (E_\bullet)^*$ such that
$$f\left( \int_0^1 \varphi_t\,dt\right) = \int_0^1 f(\varphi_t)\,dt   $$
for all integrable families $(\varphi_t \in E_\bullet \: t \in \II)$.
\end{proposition}

\begin{proof}
Without loss of generality, $E$ is nondegenerate. By Kaplansky's density theorem, each ultraweakly continuous functional on $E''$ restricts to a continuum-weakly continuous functional on $E$ of the same norm, so $(E'')_* \subsetof E_\bullet$. This implies that $\iota$ is an isometry, since certainly $\iota(x)(\varphi) \leq \|x\| \cdot \|\varphi\|$. The continuum-linearity of $\iota(x)$ for each $x \in E$ is trivial by definition of integration of continuum-weakly continuous functionals; see \cref{9}.

It remains to show that every continuum-linear functional $f \: E_\bullet \To \CC$ is in the image of $\iota$. Suppose otherwise, and consider the restriction of $f$ to $(E'')_* \subsetof E_\bullet$. Since $E''$ is canonically isomorphic to the dual $((E'')_*)^*$, we obtain an operator $x_0 \in E''$ such that $\varphi(x_0) = f(\varphi)$ for all $\varphi \in (E'')_*$. If $x_0 \not \in E$, then there is a continuum-weakly continuous functional $x \mapsto \int_0^1 \< \eta_t | x\xi_t \> \, dt$ that vanishes on $E$, but not on $x_0$. The former property can be phrased as the equation $\int_0^1 \<\eta_t| \cdot | \xi_t\> \, dt  = 0$ in $E_\bullet$. Applying the continuum-linearity of $f$, we find that 
$$ \int_0^1 \<\eta_t | x_0 \xi_t \> \, dt
= \int_0 ^1 f(\<\eta_t| \cdot |\xi_t \>) \, dt
=f \left( \int_0^1 \< \eta_t | \cdot | \xi_t \> \, dt \right) = 0. $$ This contradicts our choice of $(\xi_t)$ and $(\eta_t)$.
\end{proof}

\begin{remark}\label{28}
We briefly remark that $E$ is the dual of $E_\bullet$, if we view the latter as endowed with a variant of normed vector space structure. Let us say that a normed set $X$ is a set equipped with a function to nonnegative real numbers, which intuitively describes the size of its elements. If the normed set $X$ is equipped with a well-behaved notion of integration for functions $I \To X$, for all complex-valued measure spaces $I \prec \NN$, then $X$ is a normed vector space. If $X$ is equipped with such a notion for all $I \preccurlyeq \NN$, then $X$ is a Banach space. The space $E_\bullet$ is equipped with such a notion for all $I \preccurlyeq \RR$.
\end{remark}

\begin{lemma}\label{29}
Let $E \subsetof \B(\H)$ be a unital V*-algebra, and let $\S_\bullet \subsetof E_\bullet$ be its space of continuum-weakly continuous states. For each $x \in E_{sa}$, we define $\hat x \: \S_\bullet \To \RR$ by $\hat x (\mu) = \mu(x)$. Then, $x \mapsto \hat x$ is a bijective isometric correspondence between the self-adjoint elements of $E$, and functions $\S_\bullet \To \RR$ that are continuum-affine, i. e., that satisfy
$$f\left(\int_0^1 \mu_t \, dt\right) = \int_0^1 f(\mu_t) \, dt$$
for all families $(\mu_t \in \S_\bullet \: t \in \II)$.
\end{lemma}

\begin{proof}
We observe that $\hat x = \iota(x)|_{\S_\bullet}$, so $\hat x$ is continuum-affine. Furthermore, the function $x \mapsto \hat x$ is isometric, because the norm of a self-adjoint operator can be computed as a supremum over vector states.

If $f \: \S_\bullet\To \RR$ is continuum-affine, then it is affine in the usual sense; therefore, $f$ extends uniquely to a linear function $\tilde f\:E_\bullet \To \CC$, by \cref{11}. The canonical continuum predual $E_\bullet$ is a Banach space by \cref{10}, so $\tilde f$ is automatically bounded. By \cref{27}, to show that $f = \hat x$ for some $x \in E$, it is sufficient to show that $\tilde f$ is continuum-linear, i. e.,
$$\tilde f\left( \int_0^1 \varphi_t\,dt\right) = \int_0^1 \tilde f(\varphi_t)\,dt   $$
for each integrable family $(\varphi_t \in E_\bullet \: t \in \II)$. By \cref{11}, we may assume that each $\varphi_t$ is positive, and that $\| \int_0^1 \varphi_t \, dt\| = \int_0^1 \varphi_t(1) \, dt = 1$. The assignment $A \mapsto \int_{t\in A} \|\varphi_t\|\, dt$ is a probability measure $m$ on $\II$. If we define $\hat \varphi_t = \|\varphi_t\| \inv \varphi_t$ for nonzero $\varphi_t$, and $\hat \varphi_t = \psi$  otherwise, for some fixed $\psi \in E_\bullet$, then it remains to show 
$$ \tilde f \left(\int_{t \in T} \hat \varphi_t \, dm \right)= 
\int_{t \in T} \tilde f(\hat \varphi_t)\, dm,$$
which is just the condition that $f$ is continuum-affine on $\S_\bullet$, because $m$ is a pushforward of Lebesgue measure on $\II$. 
\end{proof}

\section{The enveloping V*-algebra}

\begin{definition}\label{40}
If $A$ is a C*-algebra, then its \underline{enveloping V*-algebra} is $V^*(A) = \overline{\gamma_\S(A)}$.
\end{definition}

\begin{lemma}\label{41}
Every continuum-weakly continuous state on $V^*(A)$ is a vector state, so we have isometric isomorphisms:
$$
\begin{tikzcd}
\S(A) \arrow[leftrightarrow]{r}{\iso} &
\S(\gamma_\S(A))  \arrow[leftrightarrow]{r}{\iso} &
\S_\bullet(V^*(A))  \arrow[leftrightarrow]{r}{\iso} &
\S_*(W^*(A)) 
\\
\mu \arrow[mapsto]{r} & 
\< \xi_\mu | \cdot \xi_\mu\> \arrow[mapsto]{r} & 
\< \xi_\mu | \cdot \xi_\mu\> \arrow[mapsto]{r} & 
\< \xi_\mu | \cdot \xi_\mu\> 
\end{tikzcd}
$$
\end{lemma}

\begin{proof}
By construction every state of $A$ factors through the universal representation $\gamma_\S$ via a vector state. Each vector state is ultraweakly continuous, and therefore extends uniquely to an ultraweakly continuous state on the ultraweak closure $W^*(A)$, and also to a unique continuum-weakly continuous state on the continuum-weak closure $V^*(A)$.
\end{proof}

\begin{lemma}\label{42}
The continuum-weak topology on $V^*(A)$ coincides with the ultraweak topology.
\end{lemma}

\begin{proof}
This is a corollary of \cref{11} and \cref{41}.
\end{proof}

\begin{proposition}\label{43}
If $A$ is unital, and $x \in W^*(A)$ is self-adjoint, then $x \in V^*(A)$ iff the function $\hat x: \S(A) \To \RR$ defined by $\hat x: \mu \mapsto \langle \xi_\mu | x \xi_\mu \rangle$ satisfies
\begin{equation*}\hat x \left( \int \mu_t \,dt \right) = \int \hat x(\mu_t)\, dt
\end{equation*}
for all indexed families $(\mu_t \in \S(A)\: t \in \II)$.
\end{proposition}

\begin{proof}
The isomorphism $\S(A) \iso \S_\bullet(V^*(A))$ in \cref{41} respects integration in the sense that $\mu(a) = \int_0^1 \mu_t(a) \, dt$ for all $a \in A$ iff $\<\xi_\mu| x |\xi_\mu\> = \int_0^1 \< \xi_{\mu_t}| x \xi_{\mu_t} \> \, dt$ for all $x \in V^*(A)$, whenever $(\mu_t\:t \in \II)$ is a family of states, because the functional $x \mapsto \<\xi_\mu| x |\xi_\mu\> - \int_0^1 \< \xi_{\mu_t}| x \xi_{\mu_t} \> \, dt$ is continuum-weakly continuous. Thus, identifying $\S(A)$ with $\S_\bullet(V^*(A))$, we can apply \cref{29} to find that $x \mapsto \hat x$ is an isometric isomorphism of $V^*(A)\sa$ onto the continuum-affine functions $\S(A) \To \RR$. Of course, $x \mapsto \hat x$ is isometric on $W^*(A)\sa$, so the proposition follows.
\end{proof}

\begin{remark}\label{44}
If $A$ is a unital \emph{separable} C*-algebra, then $a \mapsto \hat a$ is an isometric isomorphism between $A_{sa}$ and the space of continuous affine functions on $\S(A)$. Furthermore, since $\S(A) \preccurlyeq \RR$, every probability measure on $\S(A)$ is a pushforward of Lebesgue measure on $\II$, so $x \in V^*(A)$ iff $\hat x$ is strongly affine \cite{LudvikSpurny}*{section 2}.
\end{remark}

\begin{remark}\label{44.5}
In \cref{43}, we insist that $A$ be unital because the state space of a nonunital C*-algebra need not be closed under integration. For example, the state space of $c_0(\II)$ is not closed under integration because the integral of its homomorphic states over Lebesgue measure is zero.
\end{remark}

\begin{theorem}\label{45}
Let $A$ be a C*-algebra, let $E$ be a V*-algebra, and let $\rho: A \To E$ be a $*$-homomorphism. There exists a unique continuum-weakly continuous $*$-homomorphism $\pi: V^*(A) \To E$ such that $\pi\circ \gamma_\S = \rho$:
$$\begin{tikzcd}
A \arrow{r}{\gamma_\S} \arrow{rd}[swap]{\rho}
& V^*(A)\arrow[dotted]{d}{!}[swap]{\pi} \\
& E
\end{tikzcd}
$$
\end{theorem}

\begin{proof}
The universal property of the enveloping von Neumann algebra $W^*(A) = \gamma_\S(A)''$ yields an ultraweakly continuous $*$-homomorphism $\pi\: W^*(A) \To E''$ such that $\pi \circ \gamma_\S = \rho$. By \cref{13.5}, $\pi$ is also continuum-weakly continuous, which implies that $\pi(V^*(A)) \subsetof E$; thus $\pi$ makes the above diagram commute. The uniqueness of $\pi$ is trivial.
\end{proof}

\begin{remark}\label{46}
It follows from \cref{15} that every V*-algebra is closed under limits of ultraweakly convergent sequences. Therefore, $V^*(A)$ includes Davies's $\sigma$-envelope $A^\sim$ \cite{Davies68}, and Pedersen's enveloping Borel $*$-algebra $\Bb(A)$ \cite{Pedersen72}.
\end{remark}

\section{The atomic representation}

\begin{remark}\label{47}
In the Chang model, a C*-algebra may fail to have a pure state even if it has a faithful state. For example, Lebesgue measure $\lambda$ yields a faithful state on the von Neumann algebra $L^\infty(\II, \lambda)$, but it is easy to show that this C*-algebra has no pure states.
\end{remark}

\begin{remark}\label{48}
If $A$ is a \emph{separable} C*-algebra, then every state $\mu \in \S(A)$ is a mixture of pure states, in the sense that there is a probability measure $m$ on the pure state space $\partial \S(A)$ such that $\mu(a) = \int_{\nu \in \partial \S} \nu(a)\, dm$ for all $a \in A$. This is a consequence of Choquet's theorem, which yields such a Borel probability measure on $\partial\S(A)$, which of course extends uniquely to a totally defined probability measure, since $\partial\S(A)$ is Polish; see \ref{v18}, \ref{v3.2}, and \cite{Pedersen79} proposition 4.3.2. We remark that Choquet's theorem has a constructive proof \cite{Pestman99}. Clearly $\partial \S(A) \preccurlyeq \S(A) \preccurlyeq \II$, so every state on $A$ is the integral of a continuum of pure states, in the sense that there is a function $t \mapsto \varphi_t \in \partial \S(A)$ such that $\mu(a) = \int_0^1 \varphi_t(a)\, dt$ for all $a \in A$.
\end{remark}

\begin{remark}
The \underline{atomic representation} $\gamma_\partial$ of a C*-algebra $A$ is the direct sum of all of its GNS representations for pure states, i. e., $\gamma_\partial \: a \mapsto \bigoplus_{\mu \in \partial \S(A)} \gamma_\mu(a)$. If $A$ is separable, then by \cref{48}, its atomic representation is faithful.
\end{remark}

\begin{lemma}\label{49}
Let $A$ be a C*-algebra such that every state $\mu \in \S(A)$ is the integral of a continuum of states in $\partial \S(A)$. Then the continuum-weakly continuous states on $\overline{\gamma_\partial(A)}$ are in bijective correspondence with the states of $A$ via the assignment $\varphi \mapsto \varphi \circ \gamma_\partial$.
\end{lemma}

\begin{proof}
The assignment is injective because $\gamma_\partial(A)$ is trivially dense in $\overline{\gamma_\partial(A)}$. The assignment is surjective because each state of $A$ is the integral of a continuum of pure states, each of which is canonically a vector state of the representation $\gamma_\partial$.
\end{proof}

\begin{lemma}\label{50}
Let $E$ be a V*-algebra, and let $\mu:  E \To \CC$ be a continuum-weakly continuous state with GNS representation $(\gamma_\mu, \H_\mu, \xi_\mu)$. The map $\gamma_\mu$ is continuum-weakly continuous. Furthermore, if $A \subsetof E$ is a continuum-weakly dense C*-subalgebra, then $\overline{\gamma_\mu(A) \xi_\mu} = \H_\mu$.
\end{lemma}

\begin{proof}
Let $\varphi: \B(\H_\mu) \To \CC$ be a vector functional, i. e., $\varphi: y \mapsto \langle \zeta | y \eta \rangle$ for some $\eta, \zeta \in \H_\mu$. Since $\xi_\mu$ is a cyclic vector, there are sequences $(e_n)$ and $(f_n)$ from $E$ such that $\gamma_\mu(e_n) \xi_\mu \To \eta$ and $\gamma_\mu(f_n) \xi_\mu \To \zeta$. By elementary functional analysis, $$\langle \gamma_\mu(f_n) \xi_\mu| \cdot | \gamma_\mu(e_n) \xi_\mu \rangle \To \langle \zeta | \cdot | \eta \rangle,$$ in the norm topology, and therefore
$$ \langle \gamma_\mu(f_n) \xi_\mu | \gamma_\mu(\cdot) \gamma_\mu(e_n) \xi_\mu \rangle\To \langle \zeta | \gamma_\mu(\cdot) \eta \rangle $$
in the norm topology.
For each $n$, the functional $x \mapsto  \langle \gamma_\mu(f_n) \xi_\mu | \gamma_\mu(x) \gamma_\mu(e_n) \xi_\mu \rangle = \mu(f_n^*xe_n)$ is continuum-weakly continuous. Thus, $\langle \zeta | \gamma_\mu(\cdot) \eta \rangle$ is the limit of a sequence of continuum-weakly continuous functionals in the norm topology, and is therefore itself continuum-weakly continuous, by \cref{10}. We have shown that the pullback of every vector functional along $\gamma_\mu$ is continuum-weakly continuous; by \cref{13}, $\gamma_\mu$ is itself continuum-weakly continuous.

Suppose that $A \subsetof E$ is a continuum-weakly dense C*-subalgebra, and $\eta \in \H_\mu$ is a vector such that $\< \eta | \gamma_\mu(a) \xi_\mu \> =0$ for all $a \in A$. We have already established that the functional $x \mapsto \<\eta | \gamma_\mu(x) \xi_\mu\>$ is continuum-weakly continuous so it vanishes for all $x \in E$. We conclude that $\eta =0$. Therefore, $\overline{\gamma_\mu(a) \xi_\mu} = \H_\mu$.
\end{proof}

\begin{proposition}[cf. \cref{45}]\label{51}
Let $A$ be a C*-algebra such that every state $\mu \in \S(A)$ is the integral of a continuum of states in $\partial S(A)$, e. g., a separable C*-algebra. Then $\gamma_\partial$ factors uniquely through $\gamma_\S$ via a continuum-weakly homeomorphic $*$-isomorphism:
$$\begin{tikzcd}
A \arrow{r}{\gamma_\S} \arrow{rd}[swap]{\gamma_\partial}
& V^*(A)\arrow[dotted, leftrightarrow]{d}{!}[swap]{\iso}\\
& \overline{\gamma_\partial(A)}
\end{tikzcd}
$$
\end{proposition}

\begin{proof}
The enveloping V*-algebra is defined by its universal property, up to canonical isomorphism; this is established by a straightforward category-theoretic argument. In order to adapt this argument, we need only to prove the two claims diagramed below:
\begin{enumerate}
\item 
$\begin{tikzcd}
A \arrow{r}{\gamma_\partial} \arrow{rd}[swap]{\gamma_\partial}
& \overline{\gamma_\partial(A)} \arrow{d}{!}[swap]{id} \\
& \overline{\gamma_\partial(A)}
\end{tikzcd}
$
\item 
$\begin{tikzcd}
A \arrow{r}{\gamma_\partial} \arrow{rd}[swap]{\gamma_\S}
& \overline{\gamma_\partial(A)}\arrow[dotted]{d}{!} \\
& V^*(A)
\end{tikzcd}
$
\end{enumerate}
The first claim, that $\gamma_\partial$ factors through itself uniquely, via the identity, is trivial. To prove the second claim, that $\gamma_\S$ factors through $\gamma_\partial$ uniquely, we work with the universal continuum-weakly continuous representation of $\overline{\gamma_\partial(A)}$, i. e.,
$$\gamma_{\bullet}\: x \mapsto \bigoplus_{\mu \in \S_\bullet} \gamma_\mu(x),$$
where $\S_\bullet$ denotes the space of continuum-weakly continuous states $\overline{\gamma_\partial(A)} \To \CC$. Lemma \ref{49} and \cref{50}, above, now imply that the following diagram commutes:
$$
\begin{tikzcd}
A \arrow{r}{\gamma_\partial} \arrow{d}{\gamma_\S}
&
\overline{\gamma_\partial(A)} 
\arrow{d}{\gamma_{\bullet}}\\
\bigoplus_{\mu \in \S} \B(\H_\mu)
\arrow[leftrightarrow]{r}{\iso}
&
\bigoplus_{\mu \in \S_\bullet} \B(\H_\mu)
\end{tikzcd}
$$
Of course, $\gamma_{\bullet}(\overline{\gamma_\partial(A)}) \subsetof \overline{\gamma_{\bullet}(\gamma_\partial(A))}$, so taking continuum-weak closures of images we find that the following diagram commutes:
$$
\begin{tikzcd}
A \arrow{r}{\gamma_\partial} \arrow{d}{\gamma_\S}
&
\overline{\gamma_\partial(A)} 
\arrow{d}{\gamma_{\bullet}}\\
\overline{\gamma_\S(A)}
\arrow[leftrightarrow]{r}{\iso}
&
\overline{\gamma_{\bullet}(\gamma_\partial(A))}
\end{tikzcd}
$$
Thus $\gamma_\S\: A \To \overline{\gamma_\S(A)} = V^*(A)$  factors through $\gamma_\partial$ via a continuum-weakly continuous $*$-homomorphism. The uniqueness of the factorization is trivial.
\end{proof}

\begin{lemma}\label{52}
Let $T$ be a locally compact Polish space. The continuum-weak closure of $C_0(T) \subsetof \B(\ell^2(T))$ is $\ell^\infty(T)$.
\end{lemma}

\begin{proof}
Clearly, $\overline{C_0(T)} \subsetof C_0(T)'' = \ell^\infty(T)$. Let $\Bb(T)\subsetof \ell^\infty(T)$ denote the algebra of bounded Borel functions, i. e., of bounded Baire functions; these algebras are equal on any second countable locally compact Hausdorff space; see \cite{Pedersen88} proposition 6.2.9. We have the inclusion $\Bb(T)\subsetof \overline{C_0(T)}$, because V*-algebras are closed under limits of monotone sequences, by \cref{15}. If $\overline{C_0(T)} \neq \ell^\infty(T)$, then $\overline{\Bb(T)} \neq \ell^\infty(T)$, so there is a continuum-weakly continuous function $\varphi \: \ell^\infty(T) \To \CC$ such that $\varphi(\Bb(T)) = 0$, but $\varphi(p) \neq 0$ for some projection $p \in \ell^\infty(T)$. 

By \cref{39}, there is a probability measure $m$ and an integrable function $h\: T \To \CC$ such that $\varphi(f) = \int_{t \in T} f(t) h(t) \, dm$ for all $f \in \ell^\infty(T)$. By $\mathbf{LM}$ (\ref{v3.2}), there is a Borel projection $b \in \Bb(T)$ such that $b(t) = p(t)$ for almost every $t \in T$. It follows that $\varphi(p) = \int p(t) h(t)\, dm = \int b(t) h(t) \, dm = \varphi(b) = 0$, contradicting our choice of $p$. 
\end{proof}

\begin{proposition}[cf. \cref{45}]\label{53}
Let $X$ be a second countable locally compact Hausdorff space. Then $V^*(C_0(X)) \iso \ell^\infty(X)$:
$$
\begin{tikzcd}
C_0(X)\arrow[hook]{dr} \arrow{r}{\gamma_S} & V^*(C_0(X)) 
\arrow[dotted, leftrightarrow]{d}{!}[swap]{\iso}\\
& \ell^\infty(X)
\end{tikzcd}
$$
\end{proposition}
\begin{proof}
A second countable locally compact Hausdorff space is Polish, because it is an open subset of its one-point compactification, which is second countable and therefore Polish. We recall that a locally compact Hausdorff space $X$ is second countable iff $C_0(X)$ is separable, and then we apply \cref{51} and \cref{52}, above.
\end{proof}

\section{Separable C*-algebras of type I}

\begin{remark}
Let $A$ be a C*-algebra. The spectrum $\hat A$ is typically defined to be the set of unitary equivalence classes of irreducible representations of $A$. The GNS construction yields a surjective function $\partial \S(A) \To \hat A$, and for all $\mu, \nu \in \partial \S(A)$, the GNS representations $\gamma_\mu$ and $\gamma_\nu$ are unitarily equivalent iff $\mu$ and $\nu$ are unitarily equivalent, so we could also define $\hat A$ to be the set of unitary equivalence classes of pure states of $A$. If $\rho\:A \To \B(\H)$ is a type I representation, in the sense that $\rho(A)$ is nondegenerate on $\H$ and $\rho(A)''$ is a type I factor, then $\rho$ is quasiequivalent to some irreducible representation $\rho_0$, in the sense that there is a unital ultraweakly homeomorphic $*$-isomorphism $\pi\:\rho(A)_0'' \To \rho(A)''$ such that $\rho(a) = \pi(\rho_0(a))$ for all $a \in A$. Since two irreducible representations of $A$ are quasiequivalent iff they are unitarily equivalent, we could also define $\hat A$ to be the set of quasiequivalence classes of type I factor representations of $A$.

For each of the three presentations of $\hat A$ described above, we ask whether it is always possible to choose representatives from the equivalences classes that make up $\hat A$. In general, there is no choice of representative pure states $\hat A \To \partial \S(A)$, because $\hat A \succ \partial \S(A)$ whenever $A$ is a separable C*-algebra not of type I, as we show in \cref{54}. Furthermore, I do not know whether there is always a choice of representative irreducible representations. However, it is always possible to choose representative type I representations. Perhaps the simplest canonical choice of a type I representation associated to a given element of $\hat A$ is the direct sum of all the GNS representations $\gamma_\mu$ for pure states $\mu$ associated to the given element of $\hat A$, i. e., $\gamma_{[\nu]}\: a \mapsto \bigoplus_{\mu \sim\nu} \gamma_\mu(a)$ for any pure state $\nu$ associated to the given element of $\hat A$. We will sometimes equate $\hat A$ with the set $\{\gamma_{[\nu]} \: \nu \in \partial \S (A)\}$; thus, $\gamma_\partial(a) = \bigoplus_{\gamma \in \hat A} \gamma(a)$. Of course, the Hilbert space $\bigoplus_{\mu \sim \nu} \H_\mu$, on which $\gamma_{[\nu]} \in \hat A$ acts, is typically nonseparable, even when $A$ is separable. We can correct this by representing $\gamma_{[\nu]}(A)''$ on the Hilbert space of Hilbert-Schmidt operators in $\gamma_{[\nu]}(A)''$. 
\end{remark}

\begin{proposition}\label{54}
Let $A$ be a separable C*-algebra. Then, $A$ is type I iff $\hat A$ has cardinality of at most the continuum.
\end{proposition}

\begin{proof}
It is a standard theorem that a C*-algebra is type I iff the Mackey Borel structure on $\hat A$ is standard; the Mackey Borel structure consists of the sets whose preimages under the GNS construction are Borel subsets of $\partial\S(A)$. A pair of pure states produce unitarily equivalent irreducible representations iff they themselves are unitary equivalent. This a Borel equivalence relation; see \cite{Farah09} proposition 7.

If $A$ is type I, then $\hat A$ is standard, i. e., $\hat A$ equipped with the $\sigma$-algebra of Mackey Borel sets is isomorphic, as a measurable space, to a Polish space $X$ equipped with its $\sigma$-algebra of Borel sets. Clearly, $\hat A \preccurlyeq \RR$.

If $\hat A$ is not type I, then the Mackey Borel structure on $\hat A$ is not countably separated (\ref{v19}), so unitary equivalence on $\partial \S(A)$ is not Borel reducible to the identity relation on $\RR$. By the Glimm-Effros dichotomy \cite{HarringtonKechrisLouveau90}, it follows that $\hat A \succcurlyeq (\prod_{n=0}^\infty \ZZ_2 / \sum_{n=0}^\infty \ZZ_2) \succ 
\RR$.
\end{proof}

\begin{remark}\label{55}
Let $A$ be a separable C*-algebra. Pedersen has shown that the canonical map $\gamma_\S(A)'' \To \gamma_\partial(A)''$ is isometric on the enveloping Borel $*$-algebra $\Bb(A)$ \cite{Pedersen72}. Since $\Bb(A) \subsetof V^*(A)$, as in \cref{46}, the canonical map $V^*(A) \To \overline{\gamma_\partial(A)}$ is also isometric on $\Bb(A)$. We will sometimes abuse notation by identifying $\Bb(A)$ with its image, and writing $\Bb(A) \subsetof \overline{\gamma_\partial(A)}$. 
\end{remark}

\begin{remark}\label{56}
For all $\gamma \in \hat A$, the von Neumann algebra $\gamma(A)''$ is a type I factor, so an operator is in the center of $\bigoplus \gamma(A)''$ iff it is a scalar on each direct summand. In this way we obtain a $*$-isomorphism between the center of $\bigoplus \gamma(A)''$ and the von Neumann algebra of bounded functions $\hat A \To \CC$. Pedersen calls such a function Davies Borel in case it corresponds to an element of $\Bb(A)\subsetof \overline{\gamma_\partial(A)}$ under this isomorphism; see \cite{Pedersen79} section 4.7. Of course, a subset $B \subsetof \hat A$ is said to be Davies Borel iff its indicator function is Davies Borel, i. e., $(1 \:\gamma \in B;\; 0\: \gamma \in \hat A \setminus B) \in \Bb(A)$.
\end{remark}

\begin{theorem}[cf. \cref{45}]\label{57}
Let $A$ be a separable C*-algebra of type I. Then there exists a unique continuum-weakly homeomorphic $*$-isomorphism $V^*(A) \iso \bigoplus_{\gamma \in \hat A} \gamma(A)''$:
$$
\begin{tikzcd}
A \arrow{r}{\gamma_\S} \arrow{rd}[swap]{\gamma_\partial}
& V^*(A) \arrow[dotted, leftrightarrow]{d}{!}[swap]{\iso} \\
& \bigoplus \gamma(A)''
\end{tikzcd}
$$
\end{theorem}

\begin{proof}
We're proving a stronger version of \cref{51}; it remains to show that $\overline{\gamma_\partial(A)} = \bigoplus \gamma(A)''$.  Suppose that $$(x_\gamma^\infty\: \gamma \in \hat A) \in \left(\bigoplus_\gamma \gamma(A)''\right)\setminus \overline{\gamma_\partial(A)}
$$ is an operator of norm less than $1$. It follows that there is a continuum-weakly continuous functional $\varphi$ on $\bigoplus_\gamma \gamma(A)''$ such that $\varphi(\overline{\gamma_\partial(A)}) = \{0\}$, but $\varphi(x_\gamma^\infty \: \gamma) \neq 0$. Each summand $\gamma(A)''$ is canonically representable on the separable Hilbert space of Hilbert-Schmidt operators in $\gamma(A)''$, and the representation of $\bigoplus_\gamma \gamma(A)''$ on the direct sum of these Hilbert spaces yields the same continuum-weak topology as before, by an application of \cref{13}. Thus, $\varphi$ is still continuum-weakly continuous in this new representation, so \cref{39} shows that there is a family $(\varphi_\gamma\in (\gamma(A)'')_*)$ of ultraweakly continuous functionals, and a probability measure $m$ on $\hat A$ such that
$$\varphi: (x_\gamma) \mapsto \int_\gamma \varphi_\gamma(x_\gamma) \, dm  $$
for all $(x_\gamma) \in \bigoplus_{\gamma} \gamma(A)''.$

Fix a sequence $(a_n\in A_1 \: n \in \NN)$ that is dense in the unit ball $A_1 \subsetof A$. By Kaplansky's density theorem, $(\gamma(a_n) \in A)$ is dense in the unit ball of $\gamma(A)''$ for each $\gamma\in \hat A$.
Symbolically:
\begin{align*}
\forall \gamma \in \hat A \: \exists n \in \NN\:&|\varphi_\gamma(x^\infty_\gamma - \gamma(a_{n}))| \leq \epsilon \\
\exists f\: \hat A \longrightarrow \NN \: \forall^{ae} \gamma \in \hat A\: &|\varphi_\gamma(x^\infty_\gamma - \gamma(a_{f(\gamma)}))| \leq \epsilon \\
\exists f\: \hat A \mathop{\longrightarrow}^{Borel} \NN \: \forall^{ae} \gamma \in \hat A \: &|\varphi_\gamma(x^\infty_\gamma - \gamma(a_{f(\gamma)}))| \leq \epsilon
\end{align*}
The first step is a straightforward application of $\mathbf{AC_{ae}}$, and the second is a straightforward application of $\mathbf{DC}$, and $\mathbf{LM}$ in the guise of \ref{v3.2}; equipped with the Davies Borel structure, $\hat A$ is a standard Borel space, so each subset $f\inv(n) \subsetof \hat A$ differs from some Borel subset on a set of measure zero.

For each natural number $n$, let $p_n \in \ell^\infty(\hat A) \cap \Bb(A) \subsetof \ell^\infty(\hat A) \cap \overline{\gamma_\partial(A)}$ be the projection corresponding to the Davies Borel set $f\inv(n)$. Trivially, the sets $f\inv(n)$ are pairwise disjoint, and their union is $\hat A$, so $\sum_n p_n = 1$ in the ultraweak topology; similarly, the series $\sum_n \gamma_\partial(a_n) p_n$ is ultraweakly convergent because $\|a_n\| \leq 1$ for all $n$. Furthermore,
\begin{align*}
\left|\varphi\left((x^\infty_\gamma \: \gamma) - \sum_n \gamma_\partial(a_n) p_n\right)\right|
& = 
\left|
\varphi(x^\infty_\gamma - \gamma(a_{f(\gamma)})\: \gamma \in \hat A) \right|
 \\ &=
\left|\int_\gamma \varphi_\gamma(x^\infty_\gamma - \gamma(a_{f(\gamma)})) \,dm\right|
\\ & \leq
\int_\gamma |\varphi_\gamma(x^\infty_\gamma - \gamma(a_{f(\gamma)}))| \,dm
\leq \int_\gamma \epsilon \, dm = \epsilon
\end{align*}
Thus, we can ensure that $\varphi(\sum_n \gamma_\partial(a_n) p_n) \neq 0$. Of course, $\sum_n \gamma_\partial(a_n) p_n \in \overline{\gamma_\partial(A)}$ since each term is a product of elements in $\overline{\gamma_\partial(A)}$, and the series converges continuum-weakly, by \cref{15}. This contradicts our choice of $\varphi$.
\end{proof}

\section{Centralizers}

\begin{remark}\label{201}
If $A$ is a C*-algebra, then a (double) centralizer of $A$ is a pair of bounded linear maps $L, R\: A \To A$ such that $L(ab) = L(a) b$, $R(ab) = a R(b)$ and $R(a)b = a L(b)$ for all $a, b \in A$ \cite{Busby68}. Each element $a \in A$ yields a centralizer of $A$ consisting of left and right multiplication by $a$. The centralizers of $A$ form a C*-algebra $M(A) \supseteq A$, and $M(A) = A$ iff $A$ is unital.
\end{remark}

\begin{remark}\label{202}
If $A \subsetof \B(\H)$ is a nondegenerate concrete C*-algebra, then each centralizer of $A$ consists of left and right multiplication by some operator $x \in A''$ such that $xA \subsetof A$ and $Ax \subsetof A$, and $M(A)$ is isomorphic to the C*-algebra of all such operators, called the multipliers of $A$. Thus, in the full set-theoretic universe $\VVV$, these two definitions of $M(A)$ are equivalent, but in the Chang model $\CCC$, an abstract C*-algebra need not have a faithful representation. Note that if $(L, R)$ is a centralizer, then the corresponding multiplier is the ultraweak limit of $(L(e_{\lambda}))$, where $(e_\lambda)$ is any approximate identity of $A$.
\end{remark}

\begin{remark}\label{203}
If $E \subsetof \B(\H)$ is a nondegenerate V*-algebra, then its multiplier algebra $M(A)$ is also a V*-algebra, because multiplication by a fixed operator is continuum-weakly continuous. 
\end{remark}

\begin{lemma}\label{204}
Let $A$ be C*-algebra, and let $E \subsetof \B(\H)$ be a nondegenerate V*-algebra. If $\rho\: A \To M(E)$ is a $*$-homomorphism such that $\rho(A)E$ is continuum-weakly dense in $E$, then there exists a unique unital $*$-homomorphism $M(A) \To M(E)$ that extends $\rho$:
$$\begin{tikzcd}
A \arrow[hook]{r} \arrow{rd}[swap]{\rho}
& M(A) \arrow[dotted]{d}{!}[swap]{\tilde \rho} \\
& M(E)
\end{tikzcd}
$$
\end{lemma}

\begin{proof}
The proof is a minor variation on the usual argument; it focuses on the subspace $H = \mathrm{span}\{\rho(a) \eta \: a \in A, \, \eta \in \H\}$, which is dense in $\H$ because $\rho(A)E$ is continuum-weakly dense in $E$, and $E\H$ is dense in $\H$. Let $(e_\lambda)$ be an approximate unit for $A$. Clearly, $\rho(e_\lambda) \xi$ converges to $\xi$ for all $\xi \in H$, and therefore, for all $\xi \in \H$; thus, the net $(\rho(e_\lambda))$ converges to $1$ in the strong operator topology. It follows that any $*$-homomorphism $\tilde \rho$ that makes the diagram commute satisfies
$$ \tilde \rho(L, R) = \tilde \rho(L, R) \lim^{SOT}_\lambda \rho(e_\lambda) = \lim^{SOT}_\lambda \tilde \rho(L,R) \rho(e_\lambda) = \lim^{SOT}_\lambda \rho(L(e_\lambda))$$ for all $(L,R) \in M(A)$. It remains to show that such a unital $*$-homomorphism exists. 

The limit $\lim^{SOT}_\lambda \rho(L(e_\lambda))$ exists because $\rho(L(e_\lambda))\xi$ converges in norm for all $\xi \in H$; every bounded net of operators that converges pointwise on a dense subspace of a Hilbert space is convergent in the strong operator topology. Similarly, we verify that $\tilde \rho$ is a unital $*$-homomorphism via computations on the same dense subspace. For example, $\tilde \rho$ respects the adjoint operation because for all $\rho(a)\xi, \rho(b)\eta \in H$,
\begin{align*}
\< \rho(b) \eta | \tilde \rho(R^*, L^*) \rho(a) \xi\>
&= \lim_\lambda \< \rho(b) \eta | \rho(R^*(e_\lambda)) \rho(a) \xi\>
\\ &= \lim_\lambda \< \rho(R(e_\lambda) b) \eta| \rho(a) \xi\>
\\ &= \lim_\lambda \< \rho(e_\lambda L(b)) \eta | \rho(a) \xi\>
\\ &=  \< \rho(L(b)) \eta | \rho(a) \xi\>
\\ &= \lim_\lambda \< \rho(L(e_\lambda)b) \eta |\rho(a) \xi\>
\\ &= \< \tilde \rho(L, R) \rho(b) \eta | \rho(a) \xi\>.
\end{align*}
\end{proof}

\begin{lemma}\label{205}
If, in \cref{204} above, $A$ is a V*-algebra on some Hilbert space $\K$, and $\rho$ is continuum-weakly continuous, then $\tilde \rho$ is continuum-weakly continuous.
\end{lemma}

\begin{proof}
Since $A$ is a concrete C*-algebra, we may identify $M(A)$ with the algebra of multipliers of $A$, i. e., $M(A) = \{x \in A''\: xA \subsetof A, \, Ax \subsetof A\}$. For all $\eta \in \H$ and $\xi \in H$, the functional $x \mapsto \< \eta | \tilde \rho(x) \xi\>$ is continuum-weakly continuous because it is a uniform limit of the continuum-weakly continuous functionals $x \mapsto \< \eta | \rho(xe_\lambda) \xi\>$. Since $H$ is dense in $\H$, it follows that $x \mapsto \< \eta | \tilde \rho (x) \xi \>$ is continuum-weakly continuous for all $\eta, \xi \in \H$. By \cref{13}, we conclude that $\tilde \rho$ is continuum-weakly continuous.
\end{proof}

\begin{remark}\label{206}
Setting $A = E$, we find that the continuum-weak topology on $M(E)$ is representation independent.
\end{remark}

\begin{remark}\label{207}
Following Woronowicz \cite{Woronowicz79}, we define a C*-morphism from the C*-algebra $A$ to the C*-algebra $B$ to be a $*$-homomorphism $\pi\:A \To M(B)$ such that $\pi(A)B$ is norm dense in $B$. Each such $*$-homomorphism extends to a unique $*$-homomorphism $M(A) \To M(B)$, so we can compose C*-morphisms:
$$
\begin{tikzcd}
A \arrow{r}
& M(B)  \arrow[dotted]{dr}{!}
&  \\
& B \arrow{r} \arrow[hook]{u}
& M(C) 
\end{tikzcd}
$$
Appealing to \cref{205}, we obtain a similar category of V*-algebras and V*-morphisms, by defining a V*-morphism from the V*-algebra $E$ to the V*-algebra $F$ to be a continuum-weakly continuous $*$-homomorphism $\pi\: E \To M(F)$ such that $\pi(E) F$ is continuum-weakly dense in $F$.
\end{remark}

\begin{remark}\label{208}
Let $A$ be a C*-algebra and let $E$ be a V*-algebra. If $\rho\: A \To M(E)$ is a $*$-homomorphism such that $\rho(A)E$ is continuum-weakly dense in $E$, then the universal property of enveloping V*-algebras, \cref{45}, yields a unique continuum-weakly continuous $*$-homomorphism $\pi$ such that $\pi \circ \gamma_\S = \rho$:
$$
\begin{tikzcd}
A \arrow{r}{\gamma_\S} 
\arrow{rd}[swap]{\rho}
& V^*(A) \arrow[dotted]{d}{!}[swap]{\pi} \\
& M(E)
\end{tikzcd}
$$
Since, $\rho(A)E$ is continuum-weakly dense in $E$, $\pi(V^*(A))E$ is also continuum-weakly dense in $E$, so $\pi$ is a V*-morphism from $V^*(A)$ to $E$. Since $\gamma_\S(A)V^*(A)$ is continuum-weakly dense in $V^*(A)$, \cref{204} yields a unique unital $*$-homomorphism $M(A) \To M(V^*(A))$. Together, these observations show that each C*-morphism extends uniquely to a V*-morphism of enveloping V*-algebras:
$$
\begin{tikzcd}
A \arrow{r} \arrow{d}{\gamma_{\S(A)}} & M(B) \arrow[dotted]{d}{!} \arrow[hookleftarrow]{r} &  B \arrow{ld}{\gamma_{\S(B)}}\\
V^*(A) \arrow[dotted]{r}{!} & M(V^*(B)) &
\end{tikzcd}
$$
It is now straightforward, but tedious, to verify that this defines a functor from the category of C*-algebras and C*-morphisms to the category of V*-algebras and V*-morphisms.
\end{remark}

\section{Locally Polish spaces}

\begin{remark}\label{209}
Recall that a Polish space is a separable topological space whose topology comes from a complete metric. Every open subset and every closed subset of a Polish space is Polish in its subspace topology. A topological space is locally Polish in case it is Hausdorff and every point has a neighborhood that is Polish in its subspace topology. Of course, every open subset and every closed subset of a locally Polish space is locally Polish in its subspace topology.
\end{remark}

\begin{lemma}\label{210}
Every compact locally Polish space $X$ is Polish.
\end{lemma}

\begin{proof}
The space $X$ is second countable because it is the union of a finite family of open Polish subsets. By Urysohn's metrization theorem \cite{GoodTree95}, it follows that $X$ is metrizable. The space $X$ is complete for any metric that induces its topology because it is compact, and it is separable because it is second countable.
\end{proof}

\begin{lemma}\label{210.5}
Let $X$ be a locally compact locally Polish space $X$. Every compact subspace of $X$ is Polish and Baire.
\end{lemma}

\begin{proof}
Every compact subspace of $X$ is Polish as an immediate corollary of \cref{210} above. Since $X$ is regular, by the usual argument, it follows that any compact subspace $K$ has an open neighborhood $U$ whose closure is also compact, and therefore, Polish. It follows that there exists a continuous function $h\: \overline U \To \II$ such that $h \inv(1) = K$ and $h\inv (0) = \overline U \setminus U$; this function obviously extends to an element of $C_0(X)$, so $K$ is a Baire subset of $X$. 
\end{proof}

\begin{lemma}\label{211}
Let $X$ be a locally compact locally Polish space. Then every state on $C_0(X)$ is the integral of a continuum of pure states on $C_0(X)$.
\end{lemma}

\begin{proof}
Let $\mu$ be a state on $C_0(X)$. Recall that the space of bounded real-valued Baire functions on $X$ is the closure of $C_0(X)\sa$ in $\ell^\infty(X)\sa$ under pointwise limits of monotone sequences. By $\mathbf{DC}$, each bounded real-valued Baire function $b\in \Bb(X)\sa$ has a Baire code, which specifies a particular monotone sequence $(b_n)$ of bounded Baire functions of which $b$ is a limit, and particular monotone sequences for each $b_n$, etc., giving a specific construction of $b$ from elements of $C_0(X)\sa$. Given such a Baire code, we may define $\mu(b)$ by taking the corresponding limit at each stage of the construction of $b$; we can verify that such a number exists, and that it is independent of chosen Borel code. Thus, we can extend $\mu$ to the space $\Bb(X)$ of bounded complex-valued Baire functions. We obtain a probability measure $m$ on the $\sigma$-ring of Baire subsets of $X$ such that the integral of each Baire function $b \in \Bb(X)$ is equal to $\mu(b)$.

We can find a sequence of positive functions $(f_m \in C_0(X))$ of norm $1$ such that $\mu(f_m) \To 1$, so the measure $m$ is concentrated on $\bigcup_n K_n$ for some sequence $(K_n)$ of compact Baire subsets. Writing $\bigcup_n K_n = \bigcup_n (K_n \setminus \bigcup_{i <n}(K_i))$, we find that $m$ is the sum of a countable sequence $(m_n)$ of finite measures, each concentrated on the compact Baire subset $K_n$.

For each $n$, each closed subset of $K_n$ is Baire by \cref{210.5}, so the restriction of the $\sigma$-ring of Baire sets to $K_n$ is the $\sigma$-algebra of Borel sets. The subspace $K_n$ is Polish, so the restriction of $m_n$ to $K_n$ is a finite Borel measure, which extends uniquely to a totally defined measure on $K_n$ (\ref{v3.2}). Being Polish, $K_n$ injects into $\RR$, so $m_n$ is a pushforward of Lebesgue measure on $\II$ scaled by $m_n(X)$ (\ref{v3.4}). It follows that the sum $m = \sum_n m_n$ is a pushforward of the probability measure on the disjoint union of $\NN$-many copies of $\II$ obtained by scaling Lebesgue measure on the $n$-th copy by $m_n(X)$. By another application of \ref{v3.4}, we conclude that that $m$ is the pushfoward of Lebesgue measure on $\II$, i. e., $\mu$ is the the integral of a continuum of pure states.
\end{proof}

\begin{lemma}\label{211.5}
Let $X$ be a locally compact locally Polish space. If $C_0(X)$ is canonically represented on $\ell^2(X)$, then $M(\overline{C_0(X)}) = \ell^\infty(X)$. 
\end{lemma}

\begin{proof}
Clearly, $M(\overline{C_0(X)}) \subsetof C_0(X)'' = \ell^\infty(X)$. To show equality, it is sufficient to show that every function $h \in \ell^\infty(X)$ is a multiplier of $\overline{C_0(X)}$, and since mutliplication by a fixed operator is continuum-weakly continuous, it is sufficient to show that $hf \in \overline{C_0(X)}$ for all $f$ in $C_c(X)\subsetof C_0(X)$, the $*$-subalgebra of compactly supported functions. 

Therefore, suppose that $f\in C_c(X)$, and that $h \in \ell^\infty(X)$. The closed support of $f$ is Polish, by \cref{210.5}, so the set $\mathrm{supp}(f) = \{x \in X \: f(x)\neq0\}$ is also Polish. The Stone-Weierstrass theorem shows that the norm closure of $f\cdot C_0(X)$ is the direct sum of $C_0(\mathrm{supp}(f))$ in its atomic representation, and the zero C*-algebra on $\ell^2(f\inv(0))$. Applying \cref{53}, we find that
$$\ell^\infty(\mathrm{supp}(f)) \oplus 0
= \overline{C_0(\mathrm{supp}(f))} \oplus 0
= \overline{C_0(\mathrm{supp}(f)) \oplus 0}
= \overline{f \cdot C_0(X)} \subsetof \overline{C_0(X)}.$$
Of course, $\mathrm{supp}(fh) \subsetof \mathrm{supp}(f)$, so $fh \in \ell^\infty(\mathrm{supp}(f)) \oplus 0 \subsetof \overline{C_0(X)}$.
\end{proof}

\begin{theorem}\label{212}
Let $X$ be a locally compact locally Polish space. Then, there exists a unique continuum-weakly homeomorphic $*$-isomorphism $M(V^*(C_0(X))) \iso \ell^\infty(X)$ such that
$$
\begin{tikzcd}
C_0(X) \arrow{r}{\gamma_\S} \arrow[hook]{rd} & M(V^*(C_0(X)))
\arrow[dotted, leftrightarrow]{d}{!}[swap]{\iso}\\
& \ell^\infty(X).
\end{tikzcd}
$$
\end{theorem}

\begin{proof}
Let $C_0(X)$ be canonically represented on $\ell^2(X)$.
$$
\begin{tikzcd}
C_0(X) 
\arrow{r}{\gamma_\S}
\arrow[hook]{drr}
&
V^*(C_0(X))
\arrow[hook]{r}
\arrow[dotted]{rd}{!}
&
M(V^*(C_0(X)))
\arrow[dotted]{d}{!}
\\
&
&
M(\overline{C_0(X)})
\end{tikzcd}
$$
The universal property of the enveloping V*-algebra (\cref{45}) yields a unique continuum-weakly continuous $*$-morphism $V^*(C_0(X)) \To M(\overline{C_0(X)})$. Since $C_0(X)\overline{C_0(X)}$ is continuum-weakly dense in $\overline{C_0(X)}$, \cref{204} yields a unique unital continuum-weakly continuous $*$-homomorphism $M(V^*(C_0(X))) \To M(\overline{C_0(X)}) $. The morphism $V^*(C_0(X)) \To M(\overline{C_0(X)})$ is a $*$-isomorphism onto $\overline{C_0(X)}$ by \cref{211} and \cref{51}. It follows that the morphism $M(V^*(C_0(X))) \To M(\overline{C_0(X)})$ is a $*$-isomorphism, since any $*$-isomorphism of concrete C*-algebras extends to a $*$-isomorphism of their multiplier algebras. Finally, $M(\overline{C_0(X)}) = \ell^\infty(X)$ by \cref{211.5} above.
\end{proof}

\begin{corollary}\label{213}
Let $X$ be a set. Then $M(V^*(c_0(X))) \iso \ell^\infty(X)$:
$$
\begin{tikzcd}
c_0(X) \arrow{r}{\gamma_\S} \arrow[hook]{rd} & M(V^*(c_0(X)))
\arrow[dotted, leftrightarrow]{d}{!}[swap]{\iso}\\
& \ell^\infty(X)
\end{tikzcd}
$$
\end{corollary}

\appendix

\section{Proofs for \cref{cheatsheet}}\label{appendix1}

\begin{proof}[Proof of \ref{v3.2}]
All countable subsets of $X$ are Borel, so we may assume that $m$ is atomless. Apply \ref{v3.1}.
\end{proof}

\begin{proof}[Proof of \ref{v3.4}]
Without loss of generality, we assume that $(T,m)$ is atomless. Clearly $T \not \prec \NN$, so by $\mathbf{PSP}$, $T \approx \II$. Thus, it is sufficient that every atomless probability measure on $\II$ is the pushforward of Lebesgue measure on $\II$, which follows immediately from \ref{v3.1} and \ref{v3.2}.
\end{proof}

\begin{proof}[Proof of \ref{v3.5}]
By \ref{v3.4}, we may assume that $T_0, T_1 = \II$, and $m_0$ and $m_1$ are Lebesgue measure. It follows from \ref{v3.2} that $f$ is measurable, and we can verify via absoluteness that Fubini-Tonelli holds for any Lebesgue-measurable function $\II \times \II \To \CC$. 
\end{proof}

\begin{proof}[Proof of \ref{v3.73}]
Clearly, if there is a net in $Y$ converging to some point $x$, then that point is in the closure of $Y$. On the other hand, if $x$ is in the closure of $Y$, then the following net in $Y$ converges to $x$: the index set consists of all pairs $(U,y)$ where $U$ is a neighborhood of $x$, and $y$ is any point in $Y \cap U$, $(U_0, y_0) \leq (U_1,y_1)$ iff $U_0 \subsetof U_1$, and the value of the net at index $(U,y)$ is $y$.
\end{proof}

\begin{proof}[Proof of \ref{v3.75}]
If $f$ is continuous at $x$, then by definition, the preimage of every neighborhood of $f(x)$ contains a neighborhood of $x$. Thus if $(x_\lambda)$ is a net in $X$ converging to $x \in X$, then eventually $f(x_\lambda)$ is in any given neighborhood of $f(x)$. On the other hand, if $f$ is not continuous at $X$, then by definition there is a neighborhood $V$ of $f(x)$ whose preimage contains no neighborhood of $x$; in this case, $x$ is in the closure of the complement of this preimage so by \ref{v3.73} there is a net $(x_\lambda)$ converging to $x$ such that the net $f(x_\lambda)$ does not converge to $f(x)$ because it is never in $V$.
\end{proof}

\begin{proof}[Proof of \ref{v4}]
This is a known consequence of the fact that every subset of $\II$ is Lebesgue measurable \cite{Garnir74}.
\end{proof}

\begin{proof}[Proof of \ref{v5}]
We combine the open mapping theorem, verified via absoluteness, with \ref{v4}.
\end{proof}

\begin{proof}[Proof of \ref{v15}]
Without the axiom of choice, a nondegenerate representation is not necessarily the direct sum of cyclic representations; thus, the usual proof does not apply. Recall that $\gamma_\S \: A \To \B(\H_\S)$ denotes the universal representation of $A$. Fix another representation $ \rho\: A \To \B(\H)$.

For vectors $\xi, \eta \in \H$, let $\omega_{\xi, \eta}\: \gamma_\S(A) \To \CC$ be given by $\gamma_\S(a) \mapsto \langle \xi | \rho(a) \eta \rangle$, which is easily seen to be well-defined by considering the cyclic representation of $A$ on the subspace $\overline{\rho(A) \eta} \subsetof \H$. Let $\tilde \omega_{\xi, \eta}$ be its unique ultraweak extension to $W^*(A)$, the ultraweak closure of $\gamma_\S(A)$. For each operator $x \in W^*(A)$, the expression $\tilde \omega_{\xi, \eta} (x)$ defines a sesquilinear form on $\H$; it is bounded:
$$ \| \tilde \omega_{\xi, \eta}(x)\| \leq \| \tilde \omega_{\xi, \eta}\| \cdot \|x\| \leq \| \omega_{\xi, \eta}\| \cdot \|x\| \leq \|\xi\| \cdot \|\eta\| \cdot \|x\|$$
The second inequality follows by Kaplansky's density theorem.
We define $\pi(x)$ to be the unique bounded operator such that $\langle \xi | \pi(x) \eta \rangle = \tilde \omega_{\xi, \eta}(x)$. Since $\tilde\omega_{\xi, \eta}(x)$ is linear in $x$, we obtain a bounded linear map $\pi: W^*(A) \To \B(\H)$. Note that $\pi \circ \gamma_\S = \rho$, since $\< \xi | \pi(\gamma_\S(a)) \eta \> =  \tilde \omega_{\xi, \eta}(\gamma_\S(a)) = \omega_{\xi, \eta}(\gamma_\S(a)) = \< \xi | \rho(a) \eta \>$ for all $\xi, \eta \in \H$. The map $\pi$ is ultraweakly continuous since it pulls every vector functional back to a vector functional; this implies uniqueness.

It is easy to verify that $\pi$ is unital and self-adjoint. For all $\xi, \eta \in \H$,
$$ \langle \xi | \pi(1) \eta \rangle = \tilde \omega_{\xi, \eta}(1) = \omega_{\xi, \eta}(1) = \langle \xi | \eta \rangle.$$
Thus, $\pi(1) = 1$. If $x \in W^*(A)$ is self-adjoint, then we fix a net $(b_\lambda \in \gamma_\S(A)\sa)$ that ultraweakly converges to $x$, and compute for all $\xi, \eta \in \H$:
$$\langle \pi(x) \xi | \eta \rangle  = \overline{\< \eta | \pi(x) \xi\>}
= \overline{ \tilde \omega_{\eta, \xi}(x)} = \lim_\lambda \overline{\omega_{\eta, \xi} (b_\lambda)} = \lim_\lambda \langle \xi | \rho(b_\lambda) \eta \rangle = \cdots = \langle\xi | \pi(x) \eta \rangle$$

We establish that $\pi$ is a homomorphism in the usual way: we note that it is a homomorphism on $\gamma_\S(A)$, and then extend this property to $W^*(A)$ first on the first factor, and then on the second. For all $a_0, a_1 \in A$, 
$$ \pi(\gamma_\S(a_0) \gamma_\S(a_1) ) = \pi(\gamma_\S(a_0a_1)) = \rho(a_0a_1) = \rho(a_0) \rho(a_1) = \pi(\gamma_\S(a_0)) \pi(\gamma_\S(a_1))$$
For all $x \in W^*(A)$ and all $b \in \gamma_\S(A)$, there is a net $(b_\lambda \in \gamma_\S(A))$ that converges to $x$ ultraweakly, and therefore, for all $\xi, \eta \in \H$,
$$\<\xi | \pi(x) \pi(b) \eta \> 
= \lim_\lambda \< \xi | \pi(b_\lambda) \pi(b) \eta \> = \lim_\lambda \< \xi | \pi(b_\lambda b) \eta \> = \< \xi | \pi(xb) \eta \>$$
Finally, for all $x_0, x_1 \in W^*(A)$, there is a net $(b_\lambda \in \gamma_\S(A))$ that converges to $x_1$ ultraweakly, and therefore, for all $\xi, \eta \in \H$,
$$ \< \xi | \pi(x_0) \pi(x_1) \eta \> = 
\lim_\lambda \< \xi | \pi(x_0) \pi(b_\lambda) \eta \> = 
\lim_\lambda \< \xi | \pi(x_0 b_\lambda) \eta \> =
\< \xi | \pi(x_0 x_1) \eta \>.$$
\end{proof}

\begin{proof}[Proof of \ref{v28}]
A functional on $M$ is ultraweakly continuous iff it is countably additive, by \cite{Takesaki79} corollary III.3.11. A functional is countably additive iff it is countably additive on every von Neumann subalgebra of $M$ isomorphic to $\ell^\infty(\NN)$, and $\mathbf{BP}$ implies that this is the case \cite{Solovay70} \cite{Pincus74} \cite{Schechter97}*{29.37 and 29.38}. This proof was suggested to me by \textbf{Alexandru Chirvasitu}.
\end{proof}

\end{document}